\documentclass[envcountsect,reqno]{svjour3}
\usepackage{amsmath,amssymb,upgreek,lipsum}
\usepackage{fontenc,enumerate,amsfonts,hyperref,mathrsfs,amsfonts,dsfont,cite}
\usepackage{fancyhdr}
\smartqed
\hypersetup{colorlinks=true,linkcolor=red, anchorcolor=blue, citecolor=blue, urlcolor=red, filecolor=magenta, pdftoolbar=true}
\makeatletter
\newcommand{\rom}[1]{\expandafter\@slowromancap\romannumeral #1@}
\makeatother
\numberwithin{equation}{section}
\allowdisplaybreaks
\begin{document}

\title{Existence of Fractional Nonlocal Neutral Stochastic  Differential Equation of Order $1<q<2$ with Non-instantaneous Impulses and State-Dependent Delay}
\author{Surendra Kumar \and Anjali Upadhyay}

\institute{Surendra Kumar \at
              Faculty of Mathematical Sciences, Department of Mathematics, University of Delhi, New Delhi, 110007 \\
              \email{mathdma@gmail.com}  \and Anjali Upadhyay \at  Faculty of Mathematical Sciences, Department of Mathematics, University of Delhi, New Delhi, 110007\\
      \email{upadhyayanjali123@gmail.com}}
\maketitle

\begin{abstract}
This article addresses a new class of fractional nonlocal neutral stochastic differential system of order $q \in (1,2)$  including non-instantaneous impulses (NIIs) and state-dependent delay (SDD) with the Poisson jumps and the Wiener process in Hilbert spaces. We examine the existence of a mild solution without assuming the compactness of the resolvent operators. The Mönch fixed point theorem is used along with the Hausdorff measure of noncompactness. An example is constructed to verify the theoretical developments.
\end{abstract}
\keywords{Fractional Stochastic differential equation \and Fixed point \and Resolvent operator  \and  State-dependent delay  \and Hausdorff measure of noncompactness \and Mild solution}
\subclass{35A01 \and 34A37 \and 34K40 \and 60H10}

\baselineskip12pt
\section{Introduction}  Fractional delay differential equations arise in abstract formulation of numerous real world problems in the fields of science and engineering, biology and medicine. If the delay is also depending on the state variable, then it is referred as SDD \cite{Aiello&Freedman,Cao&Gard,Santos&Mallika,Hernandez&Ladeira}. It is well known that fluctuations due to noise in deterministic models affect qualitative properties (existence, uniqueness and stability) of a dynamical system. Therefore, many researchers switch  to  stochastic models instead of deterministic ones \cite{Stochastic}.
Several problems in economics, population dynamics among others are framed  by stochastic differential equations involving jump processes  which are mostly depending on the Poisson random measure \cite{Mao,Prato1992,Cont&Tankov}. However, the nonlocal conditions, initiated by  Byszewski \cite{Byszewski}, give adequate effect to the system compared to the initial condition. For more facts about the significance of nonlocal conditions one can see \cite{Ren2013,Byszewski&Lakshmikantham,Chadha&Bahuguna}.  Nowadays, neutral stochastic delay differential equations are grabbing eyeballs of many researchers due to their applications in applied mathematics \cite{Hernandez&Regan&Bala,Hernandez&Henriquez,Zhou&Jiao}. 

On the other hand, during the modelling of some problems in medical domain it is important to involve impulses. Impulses are generally of two types: instant impulses and non-instant impulses. The concept of differential equations involving instantaneous impulses has gained extensive growth in past decades \cite{Lakshmikantham&Bainov,Samoilenko&Perestyuk,Shu&Shi2016,Chadha&Bora2017,Chadha&Pandey2015}.  Deng et al. \cite{DengShu&Mao2018} considered a class of neutral stochastic functional differential equation with instant impulses driven by fBm, and established sufficient conditions for the existence of mild solutions with noncompact semigroup by utilizing the Mönch fixed point theorem and the Hausdorff measure of noncompactness. Ji and Li \cite{Ji&Li2012} studied the existence of impulsive differential equation  by applying the measure of noncompactness property along with Darbo--Sadovskii’s fixed point theorem under compactness conditions and Lipschitz conditions.

 The concept of NIIs was introduced by Hernandez and O'Regan \cite{Hernandez&Regan}. After this, several researchers studied various qualitative properties of differential equations with NIIs, and still it is an object of interest. Pierri, O'Regan and Rolnik \cite{Pierri&O'Regan}, Yu and Wang \cite{Yu&Wang}, Wang and Fečkan \cite{Michal&Wang}, Boudaoui, Caraballo and Ouahab \cite{Boudaoui&Caraballo}, Yan and Jia \cite{Yan&Jia}, Malik and Kumar \cite{Malik&Kumar} among others studied different qualitative properties of integer order differential equations with NIIs. \par       
Suganya et al. \cite{Suganya&Baleanu} investigated the existence results for a class of neutral fractional integro-differential equation involving SDD and NIIs by applying the Darbo fixed point theorem with the Hausdorff measure of noncompactness. Gautam and Dabas \cite{Gautam&Dabas2015} studied an abstract neutral fractional differential equations with NIIs and SDD and established the existence result by adopting the theory of resolvent operators and the Lerey--Schauder Alternative. Yan and Jia \cite{Yan&Jia2016}, Yan and Lu \cite{Yan&Lu2017} obtained sufficient conditions for the existence of mild solutions for fractional stochastic integro-differential equation consisting of SDD and NIIs, and also discussed the problem of optimal control. Yan and Han \cite{YanHan2019} investigated the existence of mild solution, optimal mild solution for a class of fractional partial stochastic differential systems with NIIs by using suitable fixed point theorem, and also considered the exact controllability. Using various fixed point theorems, Chaudhary and Pandey \cite{Chaudhary&Pandey} showed the existence of a solution for a class of neutral fractional stochastic integro-differential equations with SDD and NIIs.
% {The existence, controllability and other attributes of fractional differential equations are the most advancing area for research work, in particular see \cite{Byszewski,Byszewski&Lakshmikantham,Zhou&Jiao}. Shu et al. \cite{Shu&Wang2012} studied the existence and uniqueness of mild solutions for fractional differential equations with nonlocal conditions of order $1<q<2$ in Banach spaces using Krasnoselskii's fixed point theorem and solution operators. }
Inspired by the above facts and discussions, we are dealing with the existence of a mild solution for the following neutral stochastic fractional differential systems with SDD and NIIs:
\begin{equation}\label{mainequation}
\begin{cases}
\begin{aligned}
^CD _\tau^q [\xi(\tau)- b(\tau,\xi_{\varrho(\tau,\xi_\tau)})]  =& \mathscr{A}[\xi(\tau)- b(\tau,\xi_{\varrho(\tau,\xi_\tau)})] + \mathbb{J}_\tau^{2-q} [\int_{\gamma} f(s,\vartheta,\xi_{\varrho(s,\xi_s)}){\widetilde{Z}}(ds,d\vartheta)] \\ &+  \mathbb{J}_{\tau}^{1-q} [\int_{-\infty}^\tau  h(z,\xi_{\varrho(z,\xi_z)})d{\omega(z)}], \ \hspace{.5cm}  \tau\in \mathop{\cup}\limits_{i=0}^{\mathcal{N}}(s_i,r_{i+1}];
\end{aligned}\\
\xi(\tau)= l_i(\xi_{r_i})+m_i(\tau,\xi_{\varrho(\tau,\xi_\tau)}),\hspace{5.8cm} \tau\in \mathop{\cup}\limits_{i=1}^{\mathcal{N}}(r_i,s_i];\\
\xi(\tau)+k_1(\xi) =\xi_0= \uppsi(\tau), \hspace{6.5cm} \tau\in (-\infty,0];\\
[\xi(\tau)- b(\tau,\xi_{\varrho(\tau,\xi_\tau)})]_{\tau=0}^\prime +k_2(\xi)= \xi_1 \in \mathcal{V};\\
[\xi(\tau)- b(\tau,\xi_{\varrho(\tau,\xi_\tau)})]^\prime_{\tau=s_i} = 0,\hspace{6.2cm} i=1,2,\cdots,\mathcal{N}.
\end{cases}
\end{equation}
Here, $^CD_\tau^q,$ and $\mathbb{J}^q(\cdot)$ denote the Caputo fractional derivative, and the fractional integral of order $q\in (1,2)$, respectively. The operator $\mathscr{A} : D(\mathscr{A})\subset\mathcal{V} \rightarrow \mathcal{V}$ is  closed, linear and densely defined and $\xi(\cdot)$ represents the stochastic process with values in a real separable Hilbert space $(\mathcal{V},\langle\cdot,\cdot\rangle,\|\cdot ~\|)$. The history  $\xi_\tau :(- \infty,0] \rightarrow \mathcal{V}$ is given by $\xi_\tau(\varpi) = \xi(\tau+\varpi);\ \varpi \in(-\infty,0]$ and $\xi_\tau$ is the element of an abstract phase space $\mathfrak{D}$, which is explained in Sect. \ref{Sect.2}. \par Let $\mathcal{I}=[0,a],\ \mathcal{I}_0=(-\infty,0]$. The mappings $b: \mathcal{I} \times\mathfrak{D}\  \rightarrow \mathcal{V}, \ h:\mathcal{I} \times \mathfrak{D} \rightarrow  \mathbb{L}_{\mathtt{Q}}(\mathcal{K},\mathcal{V}),\  f:\mathcal{I} \times \gamma \times \mathfrak{D} \rightarrow \mathcal{V}, \ l_i : \mathfrak{D} \rightarrow \mathcal{V},\ m_i :(r_i,s_i] \times \mathfrak{D} \rightarrow \mathcal{V},\ i=1,2,\cdots, \mathcal{N}$ and $k_1, k_2 : \mathfrak{D} \rightarrow \mathcal{V}$ are appropriate functions and $\varrho: \mathcal{I} \times \mathfrak{D} \rightarrow (-\infty,a]$ is continuous. The initial data $\uppsi$ is $\mathscr{F}_0$-measurable $\mathfrak{D}$-valued, $\xi_1$ is $\mathcal{V}$-valued stochastic processes independent of the Wiener process $\omega(\cdot)$ and the Poisson point process $\tilde{z}(\cdot)$ with finite second moment.
 
 Let $0=r_0=s_0<r_1 <s_1 < \cdots <s_\mathcal{N} <r_{\mathcal{N}+1}=a$ be the impulsive points. The impulses start abruptly at $r_i$ and its effect continues on $(r_i,s_i]$. The function $\xi(\cdot)$ takes distinct values in the two subintervals $(r_i,s_i], \ (s_i,r_{i+1}]$ and is continuous at $s_i$.
 
 In this work, we consider a weak assumption for the nonlinear terms that is  Carath$\acute{e}$odory condition, also the compactness of  resolvent operators are not required for the development of our main result. we obtained a sufficient condition for the existence of a mild solution by the virtue of the Hausdorff measure of noncompactness with the M\"{o}nch fixed point theorem, theory of fractional resolvent operators and stochastic analysis.
  
A brief outline of this article is as follows: Sect. \ref{Sect.2} is devoted to notations, definitions, basic concepts of fractional calculus, phase space and the measure of noncompactness. In sect. \ref{Sect.3}, we prove the existence result for the system (\ref{mainequation}) by  utilizing the M\"{o}nch fixed point theorem. In sect. \ref{Sect.4}, an example is provided to show the feasibility of obtained result.
\section{Preliminaries and Basic Concepts}\label{Sect.2}
 We start this section by introducing the Wiener process and the Poisson point process. Let $(\Omega,\mathscr{F},\{\mathscr{F}_\tau\}_{\tau\geq 0},\mathbb{P})$ be a filtered complete probability space such that the filtration $\{\mathscr{F}_\tau\}_{\tau\geq 0}$ \ is a right continuous increasing family and $\mathscr{F}_{0}$ contains all $\mathbb{P}$-null sets. A $\mathscr{F}$-measurable function $\xi(\cdot): \Omega \rightarrow \mathcal{V}$ is called a random variable and $S = \{ \xi(\tau, \cdot): \Omega \rightarrow \mathcal{V}; \ \tau\in \mathcal{I}\}$ is called a stochastic process. We will write $\xi(\cdot)$ instead $\xi(\cdot,\omega)$. Let ${\omega(\tau)}_{\tau\geq 0}$ be $\mathcal{K}$-valued Wiener process with a finite trace nuclear covariance operator $\mathtt{Q}\geq 0$, where $\big(\mathcal{K},\langle\cdot,\cdot\rangle,\|\cdot\|\big)$ is a Hilbert space. Let $\tilde{z}={\tilde{z}(\tau)\ ; \ \tau \in D_{\tilde{z}}}$ be a stationary ${\mathscr{F}}_\tau$-Poisson point process with characteristic measure $\varkappa$. Further, let $Z(d\tau,d\vartheta)$ be the Poisson counting measure associated with $\tilde{z}$. Then $Z(\tau, \gamma) = \mathop{\sum}\limits_ {t\in D_{\tilde{z}},t\leq \tau} I_{\gamma}(\tilde{z}(t))$ with measurable set $ \gamma \in \widetilde{B}(\mathcal{K}\setminus\{0\})$ that denotes the Borel $\sigma$-field of $\mathcal{K}\setminus \{0\}$.
Let $\widetilde{Z}(d\tau,d\vartheta)\  =\ Z(d\tau,d\vartheta)- \varkappa (d\vartheta)d\tau$ be the compensated Poisson measure that is independent of $\omega(\tau)$. Let $\mathcal{M}_2(\mathcal{I}\times \gamma; \mathcal{V})$ be the space of all functions $\phi:\mathcal{I} \times \gamma \rightarrow \mathcal{V}$ such that $\int_{0}^{a} \int_{\gamma}\mathbb{E}\|\phi(\tau,\vartheta) \|^2\varkappa (d\vartheta)d\tau  < \infty$, the $\mathcal{V}$-valued stochastic integral $ \int_{0}^{a} \int_{\gamma} \phi(\tau,\vartheta)\widetilde{Z}(d\tau,d\vartheta)$ is defined, and $\phi$ is called a centred square integrable martingale. Let $\{ e_j\}_{j=1}^{\infty}$ be a complete orthonormal basis for $\mathcal{K}$. Let $\{\beta _j(\tau)\}_{j=1}^\infty $ be a sequence of real valued one-dimensional standard Brownian motions mutually independent over $(\Omega,\mathscr{F},\mathbb{P})$. For $\tau\geq0$, set $\omega(\tau) = \mathop{\sum}\limits_{j=1}^{\infty}\sqrt{ \nu _{j}} \beta_{j} (\tau) e_{j}$, where $\nu_{j} \geq 0, \ j=1,2,...$. Let $ \mathbb{L}(\mathcal{V},\mathcal{K})$ denotes the space of all bounded linear operators from $\mathcal{V}$ to $\mathcal{K}$. Let $\mathtt{Q}\in \mathbb{L}(\mathcal{K})$ be an operator defined by $\mathtt{Q}e_i = \nu_i e_i$ with $Tr(\mathtt{Q}) = \mathop{\sum}\limits_{j=1}^{\infty} \nu_j < \infty$, $Tr(\mathtt{Q})$ denotes the trace of the operator $\mathtt{Q}$. Then $\mathcal{K}$-valued  stochastic process $\omega(\tau)$ is called a $\mathtt{Q}$-Wiener process. Let  $\mathscr{F}_\tau = \sigma(\omega(t) : 0\leq t \leq \tau)$ be the $\sigma$-algebra generated by $\omega$ and $\mathscr{F}_\tau = \mathscr{F}$. For $\Theta \in \mathbb{L}(\mathcal{K},\mathcal{V})$, define $\langle\Theta, \Theta \rangle=\| \Theta\|^2_{\mathtt{Q}} = Tr(\Theta \mathtt{Q} \Theta^{\ast})= \mathop{\sum}\limits_{j=1}^{\infty}\| \sqrt{ \nu_{j}} \Theta e_{j}\|^2.$ If $\| \Theta \|_\mathtt{Q} < \infty$, then $\Theta$ is called a $\mathtt{Q}$-Hilbert--Schmidt operator. Let  $\mathbb{L}_\mathtt{Q}(\mathcal{K},\mathcal{V})$ be the space of all $\mathtt{Q}$-Hilbert--Schmidt operators $\Theta$ from $\mathcal{K}$ into $\mathcal{V}$. The completion $\mathbb{L}_{\mathtt{Q}}(\mathcal{K},\mathcal{V})$ of $\mathbb{L}(\mathcal{K},\mathcal{V}) $ with respect to the topology induced by the norm $\|\cdot\|_\mathtt{Q}$ is a Hilbert space with the above norm topology . Further, let $\mathbb{L}_{\mathscr{F}_ \tau}^2(\Omega, \mathcal{V})$ be the Banach space of all $\mathscr{F}_\tau$-measurable square integrable random variables. We introduce the space $\mathtt{PC}(\mathcal{I},\mathcal{V})$ formed by all $\mathscr{F}_\tau$-adapted measurable, $\mathcal{V}$-valued stochastic processes $\{\xi(\tau) : \tau\in\mathcal{I}\}$ such that $\xi(\tau)$ is continuous at $\tau\neq r_i, \  \xi_{r_i}= \xi_{r_{i}^- },\ \xi_{r_{i}^+ }$ exists for all $i = 1,2,...,\mathcal{N}.$\\
Also, we suppose that $\mathtt{PC}(\mathcal{I},\mathcal{V})$ is endowed with the norm $\|\xi\|_{\mathtt{PC}} = \big(\mathop{\sup}\limits _{0\leq \tau\leq a} \ \mathbb{E}\|\xi(\tau)\|^2\big)^\frac{1}{2}$, then $(\mathtt{PC}(\mathcal{I},\mathcal{V}), \|\cdot\|_{\mathtt{PC}})$ is a Banach space.

 Now we employ the definition of the phase space $(\mathfrak{D}, \|\cdot\|_{\mathfrak{D}})$ which is similar as introduced in \cite{Hino,Hale&Kato1978}. More accurately, $\mathfrak{D}$ is a semi-normed linear space of $\mathscr{F}_0$-measurable mappings from $\mathcal{I}_0$ into $\mathcal{V}$ and satisfying the subsequent axioms:\\
\rom{1}.\hspace{0.5cm} If $\xi: \ (-\infty, c+a] \rightarrow \mathcal{V},\  a>0$, is such that $\xi|_{[c,c+a]} \in  \mathtt{PC}([c,c+a],\mathcal{V})$ and $\xi_c \in \mathfrak{D}$, thereafter for every $\tau \in [c,c+a]$ the following  hold:
\begin{enumerate}
\item $\xi_\tau \in \mathfrak{D}$; \
 \item $\|\xi(\tau)\| \leq N_1 \|\xi_\tau\|_{\mathfrak{D}}$, where $N_1>0$ is a constant;\
 \item $\|\xi_\tau\|_{\mathfrak{D}} \leq \  N_2(\tau-c) \sup \{\|\xi(t)\|:\ c\leq t \leq \tau\} + N_3(\tau-c)\|\xi_c\|_\mathfrak{D}$,\\ where $N_2, N_3 : \mathds{R}^+\cup\{0\} \rightarrow [1,\infty)$, $N_2$ is  continuous, $N_3$ is  locally bounded and independent of $\xi(\cdot)$.
\end{enumerate}
%\rom{2}. \hspace{.5cm} For  $\xi(\cdot)$ in \rom{1}, the map $t \rightarrow \xi_t$ is continuous  from $[c,c+a]$ into $\mathfrak{D}$\\
\rom{2}.\hspace{.5cm} The space $\mathfrak{D}$ is complete.\\
The following result is a consequence of the phase space axioms:
\begin{lemma} \label{lemma2.1}\cite{Yan&Jia}
Let $\xi : (-\infty, a] \rightarrow \mathcal{V}$  be an $\mathscr{F}_\tau$-adapted measurable process such that \\$\xi|_{\mathcal{I}} \in \mathtt{PC}(\mathcal{I},\mathcal{V})$, and $\xi_0 = \uppsi(\tau) \in \mathbb{L}_{\mathscr{F}_0}^2(\Omega, \mathfrak{D})$, then
\begin{center}
$\|\xi_\tau\|_{\mathfrak{D}} \leq \ N_2^{\star}\mathop{\sup}\limits_{t\in  \mathcal{I}} \mathbb{E}\|\xi(t)\|+N_3^{\star}\|\uppsi\|_{\mathfrak{D}}$,
\end{center}
where $ N_2^{\star} = \mathop{\max}\limits_{t\in \mathcal{I}} N_2(t)$ and $ N_3^{\star} =  \mathop{\sup}\limits_{t\in \mathcal{I}} N_3(t)$.

\end{lemma}
\begin{definition}\cite{Kilbas&Trujillo,Podlubny}
For $g \in L^1((0,a), \mathds{R})$, $0<a < \infty$ the fractional integral of order $q>0$ with lower limit 0  is given by
\begin{align*}
\mathbb{J}_\tau^q g(\tau)= \frac{1}{\Gamma(q)}\int_0^\tau (\tau-t)^{q-1} g(t)dt.
\end{align*} 
\end{definition}
\begin{definition}\cite{Kilbas&Trujillo,Podlubny}
For $g \in C^{m-1}((0,a), \mathds{R})\cap \mathbb{L}^1((0,a),\mathds{R})$, $0<a < \infty$ the Caputo derivative of order $q>0$ with lower limit 0 is defined by
\begin{align*}
^cD_\tau^q g(\tau) = \frac{1}{\Gamma(m-q)}\int_0^\tau (\tau-t)^{m-q-1} g^{(m)}(t)dt,\ 0<m-1\leq q<m,\ m\in\mathbb{N}.
\end{align*}
\end{definition}
\begin{remark}
 If g is $\mathcal{V}$-valued, the integral in Definitions 2.1 and 2.2 is taken in Bochner's sense.
\end{remark}
\begin{definition}\cite{Shu&Wang2012}
A closed linear operator $\mathscr{A}: D(\mathscr{A})\subset \mathcal{V} \rightarrow \mathcal{V}$ is called a sectorial operator of type $(\mathrm{M}^*,\uptheta,q,\Lambda)$ if there are constant $\Lambda\in \mathds{R},\ 0<\uptheta < \pi/2,\ \mathrm{M}^\ast\in \mathds{R}^+$ such that
\begin{enumerate}
\item The $q$-resolvent of $\mathscr{A}$ exists outside the sector $\Lambda+S_\uptheta= \{\Lambda +\lambda ^q :\ \lambda \in \mathbb{C},\ |Arg(-\lambda^q)|<\uptheta\}$,\
\item $\|\mathsf{R}(\lambda^q, \mathscr{A})\| = \|(\lambda^q-\mathscr{A})^{-1}\| \leq \frac{\mathrm{M}^\ast}{|\lambda^q-\Lambda|},\ \lambda^q \not\in \Lambda+S_\uptheta$.
\end{enumerate}
If $\mathscr{A}$ is the sectorial operator of type $(\mathrm{M}^\ast,\uptheta,q,\Lambda)$, then $\mathscr{A}$ generates a $q$-resolvent family $\{\mathcal{R}_q(\tau)\}_{\tau\geq 0}$, where $\mathcal{R}_q(\tau) =  \frac{1}{2\pi i}\int_{\mathbf{C}} e^{\lambda \tau} \mathsf{R}(\lambda^q, \mathscr{A}) d\lambda$, and $\mathbf{C}$ is a suitable path such that $\lambda^q \not\in \Lambda+S_\uptheta$.
\end{definition}
\begin{lemma}\cite{Shu&Shi2016}
Suppose that $\mathscr{A}$  is the sectorial operator of type $(\mathrm{M}^*,\uptheta,q,\Lambda)$, and $g$ satisfies the uniform Holder condition with the exponent $\beta \in (0,1]$. Then the unique solution of the Cauchy problem
\begin{align*}
^CD_\tau^q \xi(\tau)&= \mathscr{A}\xi(\tau) + \mathbb{J}_\tau^{2-q}g(\tau), \ 1<q<2,\ \tau \in \mathcal{I}\\
\xi(0)&= \xi_0 \in \mathcal{V},\ \xi^\prime(0)= \xi_1 \in \mathcal{V},
\end{align*}
is given by
\begin{align*}
\xi(\tau)= \mathcal{T}_q(\tau)\xi_0+\mathcal{S}_q(\tau)\xi_1+\int_0^\tau \mathcal{S}_q(\tau-t)g(t)dt, \  \tau\in\mathcal{I},
\end{align*}
where
\begin{align*}
\mathcal{T}_q(\tau) = \frac{1}{2\pi i}\int_\mathbf{C} e^{\lambda \tau} \lambda^{q-1}\mathsf{R}(\lambda^q, \mathscr{A})d\lambda;\ 
\mathcal{S}_q(\tau) = \frac{1}{2\pi i}\int_\mathbf{C} e^{\lambda \tau} \lambda^{q-2}\mathsf{R}(\lambda^q, \mathscr{A})d\lambda.
\end{align*}
\end{lemma}
\begin{definition}\cite{D.K1985,J&K1980}
The Hausdorff  measure of noncompactness, denoted by $\mu$, of a bounded set $B \neq \emptyset , \ B\subset\mathcal{V}$ is defined as
$\mu(B) = \inf \{\epsilon > 0 :$ B has a finite $\epsilon$--net in$\ \mathcal{V} \}$.
\end{definition}

Some fundamental properties of the Hausdorff measure of
noncompactness are listed below:
\begin{lemma}\label{lemma2.2}(\cite{Ji&Li2012,D.K1985,J&K1980}) 
For the proper bounded sets $U,U_1,U_2 \subset\mathcal{V}$, we have
\begin{enumerate}
\item  $U_1 \subseteq U_2 \  \Rightarrow \mu(U_1) \leq \mu(U_2)$.\
\item $\mu(U)= \mu(\overline{U})=\mu(\textbf{co}(U))$, where $\overline{U}$ and $\textbf{co}(U)$ are closure and convex hull of $U$, respectively;\
\item  $\overline{U}$ is compact $\Leftrightarrow$ $\mu(U) = 0$;\
\item $\mu(U_1 +U_2) \leq \mu(U_1)+\mu(U_2)$;\
\item $\mu(rU) \leq |r|\ \mu(U)$, $r\in \mathds{R}$;\
%\item  $\mu(\{\upsilon\}\cup \mathbf{B}) = \mu(\mathbf{B})$, for every $\upsilon \in \mathcal{V}$ and every non-empty bounded set $\mathbf{B}_ \subset \mathtt{PC}(\mathcal{I},\mathcal{V})$;\
\item  If $U \subset \mathtt{PC}(\mathcal{I},\mathcal{V})$ is bounded and equicontinuous, then $t \mapsto \mu(U(t))$ is piecewise continuous on $\mathcal{I}$ and $\mu(U) =\  \mathop{\sup}\limits_{s\in \mathcal{I}} \mu(U(s))$. Moreover, for every $\tau\in \mathcal{I},\ \mu\Big(\int_0^\tau U(s)ds\Big)\leq \int_0^\tau \mu(U(s))ds$,\\
where $U(s) = \{\xi(s) : \xi \in U\}$ and $\int_0^\tau U(s)ds = \Big\{\int_0^\tau \xi(s)ds:\ \xi\in U \Big\};$
%\end{align*}
\item Let $\{\xi_n\}_{n=1}^\infty$ be a sequence of Bochner's integrable functions from $\mathcal{I}$ into $\mathcal{V}$ such that \\$\|\xi_n(\tau)\| \leq \ \hat q(\tau)$ for each $n \geq 1$ and almost all $\tau\in \mathcal{I},$ where $\hat q \in L^1(\mathcal{I},\mathds{R}^+ )$, then \\$\hat{g}(\cdot) = \mu \{\xi_n(\cdot) : n \geq 1\}\in L^1(\mathcal{I},\mathds{R}^+ )$ satisfies
\begin{center}
$\mu (\{\int_0^\tau \{\xi_n(t)\}_{n\geq 1}dt\}) \leq \ 2\int_0^\tau \hat{g}(t)dt$.
\end{center}
\end{enumerate}
\end{lemma}
In case of stochastic integral, we use the following result:
\begin{lemma} \label{lemma2.3}\cite{DengShu&Mao2018}
 If $H \subset C(\mathcal{I};\ \mathbb{L}_\mathtt{Q} (\mathcal{V},\mathcal{K}))$ and $\omega$ is the standard Winer process, then 
\begin{center}
$\mu(\int_0^\tau H(t)d\omega(t)) \leq \sqrt{a} \sqrt{Tr(\mathtt{Q})}\  \mu (H(t))$,
 \end{center}
where $\int_0^\tau H(t) d\omega(t) = \{ \int_0^\tau v(t)d\omega(t)  : \ v \in H, \ \tau \in \mathcal{I}\}$.
 \end{lemma}
% \begin{lemma}\label{lemma2.5} \cite{Prato1992}
% For any $p\geq 1$ and for arbitrary $\mathbb{L}_0^2(\mathcal{V},\mathcal{K})$-valued predictable process $W(\cdot)$,
% \begin{align*}
% \mathop{\sup}\limits_{s\in[0,t]}\mathbb{E}\|\int_0^s W(s)d\omega(s)\|^{2p}\  \leq \ \big(p(2p-1)\big)^p \Big(\int_0^t \big(\mathbb{E}\|W(s)\|_{\mathbb{L}_0^2}^{2p}\big) ds\Big)^p,\ \ t \in [0,\infty),
% \end{align*}
% \end{lemma}
 \begin{lemma}\label{lemma2.4} \cite{Monch1980}
Suppose $\mathbf{B}\subset\mathcal{V}$ is convex, closed and $\ 0\in \mathbf{B}$. If the map $\Gamma : \mathbf{B} \rightarrow \mathcal{V}$
is continuous and $\Gamma$ satisfies the property:
\begin{center}
$\mathbf{G} \subset \mathbf{B}, \ \mathbf{G}$ is countable, $\mathbf{G}\subset \overline{\textbf{co}}\big(\{0\} \bigcup \Gamma (\mathbf{G})\big)
\Rightarrow \ \overline{\mathbf{G}}$ is compact,
\end{center}
then $\Gamma$ admits a fixed point in $\mathbf{B}$.
\end{lemma}
In virtue of Lemma \ref{lemma2.2}, mild solution of system (\ref{mainequation}) is defined as
\begin{definition}
 An $\mathscr{F}_\tau$-adapted stochastic process $\xi : (-\infty, a] \rightarrow \mathcal{V}$ is called a mild solution of (\ref{mainequation}) if 
\begin{enumerate}              
\item $\xi_\tau$ is measurable and adapted to $\mathscr{F}_\tau,\ \tau \geq 0$,\
\item $\xi(\tau) \in \mathcal{V}$ has $c\acute{a}dl\acute{a}g$ paths on $\tau\in\mathcal{I}$ a.s.,\
\item $\ \xi_0= \uppsi, \ \xi_{\varrho(s,\xi_s)}\in \mathfrak{D}$; $s\in \mathcal{I}$ with $\xi_{1}  \in  \mathbb{L}_{\mathscr{F}_0}^2(\Omega, \mathcal{V}),\ \xi|_{\mathcal{I}} \in \mathtt{PC}(\mathcal{I},\mathcal{V})$ and following integral equation hold:
\begin{equation}\label{mildsolution}
\xi(\tau)=
\begin{cases}
%\begin{aligned}
 \mathcal{T}_q(\tau)[ \uppsi(0)- k_1(\xi) - b(0,\uppsi)]+\mathcal{S}_q(\tau)[\xi_1-k_2(\xi)] \\+ b(\tau,\xi_{(\varrho(\tau,\xi_\tau)})+ \int_0^\tau \int_{\gamma} \mathcal{S}_q (\tau-s) f(s,\vartheta,\xi_{\varrho(s,\xi_s)})\widetilde{Z}(ds,d\vartheta)\\+ \int_0^\tau \mathcal{T}_q (\tau-s) (\int_{-\infty}^s h(z, \xi_{\varrho(z,\xi_{z})})d\omega(z))ds,\hspace{3.1cm} \tau \in [0,r_1];\\
 l_i(\xi_{r_i}) + m_i(\tau, \xi_{\varrho(\tau,\xi_\tau)}),\hspace{6cm} \tau\in \mathop{\cup}\limits_{i=1}^{\mathcal{N}}(r_i,s_i]; \\
 \mathcal{T}_q(\tau-s_i)[ l_i(\xi_{s_i}) + m_i(s_i, \xi_{\varrho(s_i,\xi_{s_i})}) - b(s_i,\xi_{\varrho(s_i,\xi_{s_i})})]\\+ b(\tau,\xi_{(\varrho(\tau,\xi_\tau)})
+ \int_{s_i}^\tau \int_{\gamma} \mathcal{S}_q (\tau-s) f(s,\vartheta,\xi_{\varrho(s,\xi_s)}) \widetilde{Z}(ds,d\vartheta) \\+ \int_{s_i}^\tau \mathcal{T}_q (\tau-s) (\int_{-\infty}^s h(z, \xi_{\varrho(z,\xi_{z})})d\omega(z))ds,\hspace{2.4cm} \tau\in \mathop{\cup}\limits_{i=1}^{\mathcal{N}}(s_i,r_{i+1}].
%\end{aligned}
\end{cases}
\end{equation}
\end{enumerate}
\end{definition}
\section{Existence of solution}\label{Sect.3}
 For further development, we construct the following hypothesis:
\begin{itemize}
\item[$(\textbf{S}_\uppsi):$] Let $\mathsf{R}(\varrho^-)=\{\varrho(\tau,\varphi)\leq 0:(\tau,\varphi)\in \mathcal{I}\times\mathfrak{D}\}.$ The map $\tau \mapsto \uppsi_\tau$ is continuous from $\mathsf{R}(\varrho^-)$ into $\mathfrak{D}$ and there is a bounded continuous map $J^{\uppsi}:\mathsf{R}(\varrho^-)\rightarrow \mathds{R}^+$ such that $\|\uppsi_\tau\|_\mathfrak{D}\leq J^{\uppsi}(\tau)\|\uppsi\|_\mathfrak{D}$, for every $\tau\in \mathsf{R}(\varrho^-)$. 
\item [$(\textbf{S}_1):$] The linear operators $\mathcal{T}_q(\cdot)$ and $\mathcal{S}_q(\cdot)$ are strongly continuous, and there is $\mathrm{M}>0$ such that
\begin{align*}
\mathop{\sup}\limits_{\tau\in \mathcal{I}}\|\mathcal{T}_q(\tau)\| \ \vee \ \mathop{\sup}\limits_{\tau\in \mathcal{I}}\|\mathcal{S}_q(\tau)\|\ \leq \mathrm{M}.
\end{align*}
\item [$(\textbf{S}_2):$] (i) The mappings $k_1, k_2 :\mathfrak{D} \rightarrow \mathcal{V}$ are continuous and there are constants $L_{k_1}>0$ and $L_{k_2}>0$ such that
\begin{align*}
%\mathbb{E}\|k_1(\xi_1)-k_1(\xi_2)\|^2\ \leq \ L_{k_1} \|\xi_1-\xi_2\|_{\mathfrak{D}}^2, \ \ \ \xi_1,\xi_2 \in \mathfrak{D},\\
\mathbb{E}\|k_1(\xi)\|^2\ \leq \ L_{k_1}(1+\|\xi\|_{\mathfrak{D}}^2) ~\mbox{and}~
\mathbb{E}\|k_2(\xi)\|^2\ \leq \ L_{k_2}(1+ \|\xi\|_{\mathfrak{D}}^2);~\xi\in\mathfrak{D}.
\end{align*}
\item[(ii)] There are constants $l_1^*>0$ and $l_2^*>0$ such that for bounded sets $D_1,D_2 \subset \mathfrak{D}$,
\begin{align*}
\mu(k_{1}(D_1))\ \leq\ l_1^* \mathop{\sup}\limits_{\theta \in \mathcal{I}_0} \mu (D_1(\theta)),\ \text{and} \   
\mu(k_{2}(D_2))\ \leq\ l_2^* \mathop{\sup}\limits_{\theta \in \mathcal{I}_0} \mu (D_2(\theta)).
\end{align*}
\item[$(\textbf{S}_3):$]  The mapping $b: \mathcal{I} \times \mathfrak{D} \rightarrow \mathcal{V}$ satisfies following:
\item[(i)] Let $\xi:(-\infty,a]\rightarrow\mathcal{V}$ be a map such that $\xi_0=\uppsi$ and $\xi|_\mathcal{I} \in \mathtt{PC}$. The map $\tau \mapsto b(s,\xi_\tau)$ is continuous on $\mathsf{R}(\varrho^-)\cup \mathcal{I}$ for every $s\in \mathcal{I}$ and $b(\tau,\cdot):\mathfrak{D}\rightarrow \mathcal{V}$ is continuous for every $\tau\in \mathcal{I}$. 
\item[(ii)] There exists $M_b>0$ such that
\begin{align*}
\mathbb{E}\|b(\tau,\xi)\|^2\ \leq \ M_b (1+\|\xi\|_{\mathfrak{D}}^2), \ \ \ (\tau,\xi) \in \mathcal{I}\times\mathfrak{D}
%\mathbb{E}\|b(\tau,\xi_1)-b(\tau,\xi_2)\|^2\ \leq \ M_b \|\xi_1-\xi_2\|_{\mathfrak{D}}^2, \ \ \ \tau\in \mathcal{I}, \  \ \xi_1, \xi_2 \in \mathfrak{D}.  
\end{align*}
\item[(iii)] There is a function $l_b(\cdot) \in L^1(\mathcal{I},\mathds{R}^+)$ such that for any bounded set $D_3 \subset \mathfrak{D}$,
\begin{align*}
\mu(b(\tau,D_3)) \leq l_b(\tau)\mathop{\sup}\limits_{\theta \in \mathcal{I}_0} \mu (D_3(\theta)),\ \text{and} \ l_b^* = \mathop{\sup}\limits_{\tau \in \mathcal{I}}l_b(\tau).\
\end{align*}
\item[$(\textbf{S}_4):$]For $m_i : (r_i,s_i]\times \mathfrak{D} \rightarrow \mathcal{V}$ the following hold: 
\item[(i)] The map $\tau\mapsto m_i(s,\xi_\tau),\ 1\leq i \leq \mathcal{N}$, are continuous on $\mathsf{R}(\varrho^-)\cup \mathcal{I}$ for every $s\in \mathcal{I}$, and $m_i(\tau,\cdot):\mathfrak{D}\rightarrow \mathcal{V}$ is continuous for every $\tau\in \mathcal{I}$.
%\item[(ii)] $\mathbb{E}\|m_i(\tau,\xi_1)-m_i(\tau,\xi_2)\|^2\ \leq \ M_i \|\xi_1-\xi_2\|_{\mathfrak{D}}^2, \ \ \ \tau\in (r_i,s_i], \   \ \xi_1,\xi_2 \in \mathfrak{D}$.\
%\end{align*}
\item[(ii)] There are constants $M_i>0, \  1\leq i \leq \mathcal{N}$, such that
\begin{align*}
 \mathbb{E}\|m_i(\tau,\xi)\|^2\ \leq \ M_i (1+\|\xi\|_{\mathfrak{D}}^2), \ \  (\tau,\xi)\in (r_i,s_i]\times \mathfrak{D}.
 \end{align*}
\item[(iii)] For any bounded set $D_4 \subset \mathfrak{D}$ there exist $l_{m_i}(\cdot) \in L^1(\mathcal{I},\mathds{R}^+)$  to ensure that
\begin{align*}
\mu(m_{i}(\tau,D_4))\ \leq\ l_{m_i}(\tau)\mathop{\sup}\limits_{\theta \in \mathcal{I}_0} \mu (D_4(\theta)),\ \text{and}\ l_{m_i}^*= \mathop{\sup}\limits_{\tau\in\mathcal{I}}l_{m_i}(\tau).
\end{align*}
\item[$(\textbf{S}_5):$] (i)  The functions $l_i : \mathfrak{D} \rightarrow \mathcal{V}$ are continuous and there are continuous non-decreasing functions $\mathcal{A}_{l_i}:\mathds{R}^+\cup\{0\} \rightarrow \mathds{R}^+,\  1\leq i \leq \mathcal{N}$, such that
\begin{align*}
\mathbb{E}\|l_i(\xi)\|^2\ \leq \ \mathcal{A}_{l_i} (\|\xi\|_{\mathfrak{D}}^2), \ \  \xi\in \mathfrak{D},\ \text{and}\ 
\mathop{\liminf}\limits_{r\rightarrow \infty} \frac{ \mathcal{A}_{l_i}(r)}{r}\ =\ \lambda_i \ < \infty.
\end{align*}
\item[(ii)] There exist $L_i>0,\  1\leq i \leq \mathcal{N}$, such that for any bounded set $D_5 \subset \mathfrak{D}$, 
\begin{align*}
\mu(l_{i}(D_5))\ \leq\ L_i \mathop{\sup}\limits_{\theta \in \mathcal{I}_0} \mu (D_5(\theta)).
\end{align*}
\item[$(\textbf{S}_6):$] The function $h: \mathcal{I} \times \mathfrak{D} \rightarrow \mathbb{L}_\mathtt{Q}(\mathcal{K},\mathcal{V})$ satisfies the following:
\item[(i)] The mapping $h(\tau,\cdot): \mathfrak{D} \rightarrow \mathbb{L}_\mathtt{Q}(\mathcal{K},\mathcal{V})$ is continuous for a.e. $\tau\in \mathcal{I}$, and  $h(\cdot,\xi): \mathcal{I}  \rightarrow \mathbb{L}_\mathtt{Q}(\mathcal{K},\mathcal{V})$ is strongly measurable for every $\xi \in \mathfrak{D}$.\
\item[(ii)] There exists $\mathfrak{m}(\cdot) \in L^1(\mathcal{I},\mathds{R}^+)$ and a continuous non-decreasing function \\$\mathcal{A}_h :\mathds{R}^+\cup \{0\} \rightarrow \mathds{R}^+$ such that for $(z,\xi) \in \mathcal{I} \times \mathfrak{D}$, 
\begin{align*}
\int_0^\tau \mathbb{E}\|h(z,\xi)\|_\mathtt{Q}^2 dz \ \leq \ \mathfrak{m}(\tau) \mathcal{A}_h (\|\xi\|_{\mathfrak{D}}^2), \ \ 
\mathop{\liminf}\limits_{r\rightarrow \infty} \frac{ \mathcal{A}_h(r)}{r}\ =\ \lambda_h \ < \infty.
\end{align*}
\item[(iii)] For each $\xi \in \mathfrak{D}, \ k(\tau)= \mathop{\lim}\limits_{e\rightarrow \infty} \int_{-e}^0 h(\tau,\xi)d\omega(z)$ exists, $k(\cdot)$ is continuous, and $\mathbb{E}\|k(\cdot)\|^2 \leq M_h$, for some  constant $M_h>0$.
%\item[(iv)] There exists constant $L_h>0$ such that 
%\begin{align*}
%\int_0^t \mathbb{E}\|h(t,\xi_1)-h(t,\xi_2)\|^2 \ \leq \ L_h \|\xi_1-\xi_2\|_{\mathfrak{D}}^2, \ \ \ t\in \mathcal{I}, \ \ \xi_1,\xi_2 \in \mathfrak{D}.  
%\end{align*}
\item[(iv)] There exists $l_h(\cdot) \in L^1(\mathcal{I},\mathds{R}^+)$ such that
%\begin{align*}
$\mu(h(\tau,D_6))\ \leq\ l_h(\tau)\mathop{\sup}\limits_{\theta \in \mathcal{I}_0} \mu (D_6(\theta))$, where $D_6 \subset \mathfrak{D}$ is bounded.
%\end{align*}
\item[$(\textbf{S}_7):$](i) The map $t\mapsto f(\tau,\cdot,\xi_{\varrho(\tau,\xi_\tau)})$ is Borel measurable on $\mathcal{I}$, and $\tau\mapsto f(s,\vartheta,\xi_\tau)$ is continuous on $\mathsf{R}(\varrho^-)\cup \mathcal{I}$ for every $s\in \mathcal{I}$.\
%\begin{align*}
% \item[(ii)]  There exists constant $M_f >0$ such that \begin{center}$\int_{\gamma} \mathbb{E}\|f(\tau,\vartheta,\xi)\|^2 \varkappa d\vartheta \ \leq \  M_f (1+\|\xi\|_{\mathfrak{D}}^2), \ \ \ \tau\in \mathcal{I}, \ \ \xi\in \mathfrak{D}$.\end{center}
%\item[(ii)] $\int_{\gamma} \mathbb{E}\|f(\tau,\vartheta,\xi_1)-f(\tau,\vartheta,\xi_2)\|^2 \varkappa d\vartheta \ \leq \  M_f \|\xi_1-\xi_2\|_{\mathfrak{D}}^2, \ \ \ \tau\in \mathcal{I}, \  \ \xi_1,\xi_2 \in \mathfrak{D}$.
\item[(ii)] There exist $\mathfrak{n}(\cdot) \in L^1(\mathcal{I},\mathds{R}^+)$, and a continuous non-decreasing function $\mathcal{A}_f :\mathds{R}^+\cup\{0\} \rightarrow \mathds{R^+}$ to ensure that 
%\begin{align*}
$ \int_\gamma \mathbb{E}\|f(\tau,\vartheta,\xi)\|^2 \varkappa(d\vartheta) \ \leq \ \mathfrak{n}(\tau) \mathcal{A}_f (\|\xi\|_{\mathfrak{D}}^2), \ \ 
\mathop{\liminf}\limits_{r\rightarrow \infty} \frac{ \mathcal{A}_f(r)}{r}\ =\ \lambda_f \ < \infty$.
%\end{align*}
\item[(iii)] There is $l_f(\cdot,\cdot)\in L^1(\mathcal{I}\times \gamma,\mathds{R}^+)$ such that
%\begin{align*}
$\mu(f(\tau,\vartheta,D_7))\ \leq\ l_f(\tau,\vartheta)\mathop{\sup}\limits_{\theta \in \mathcal{I}_0} \mu (D_7(\theta)),$ where $D_7 \subset \mathfrak{D}$ is  bounded.\\
%\end{align*}
For convenience, let $\mho=\mathop{\sup}\limits_{\tau \in \mathcal{I},\vartheta \in \gamma} \int_0^\tau \int_\gamma l_f(s,\vartheta)\widetilde{Z}(ds,d\vartheta)<\infty$.
%\end{align*}
\end{itemize}
To deal with term containing SDD, we apply the following:
\begin{lemma}\label{lemma3.1}(Compare with \cite{Hernandez&Ladeira})
Let $\xi : (-\infty, a] \rightarrow \mathcal{V}$  be an $\mathscr{F}_\tau$-adapted measurable process such that $\xi_0 = \uppsi(\tau) \in \mathbb{L}_{\mathscr{F}_0}^2(\Omega, \mathfrak{D})$
and $\xi|_{\mathcal{I}} \in \mathtt{PC}(\mathcal{I},\mathcal{V})$.
 If $(S_\uppsi)$ holds, then

$\|\xi_\tau\|_{\mathfrak{D}} \leq \ N_2^{\star}\sup \{ \mathbb{E}\|\xi(\tau)\|: \ \tau \in [0,max\{0,s\}],\  s \in \mathsf{R}(\varrho^-)\cup \mathcal{I}\}+(N_3^{\star} + J^{\star})\|\uppsi\|_{\mathfrak{D}}$,\\
where $J^{\star} = \mathop{\sup}\limits_{s\in \mathsf{R}(\varrho^-)} \ J^\uppsi (s)$.
\end{lemma}
Consider the space $\mathfrak{D}_a = \{ \xi:(-\infty,a] \rightarrow \mathcal{V}:\ \xi_0 \in \mathfrak{D},\ \xi_1\in\mathcal{V},\  \xi|_{\mathcal{I}}\in \mathtt{PC}\big(\mathcal{I},\mathcal{V}\big),\ \mathop{\sup}\limits_{\tau\in \mathcal{I}}\mathbb{E}\|\xi(\tau)\|^2 < \infty \}$ with the semi-norm
%\begin{align*}
$\|\xi\|_a= {\|\xi_0 \|}_\mathfrak{D}+\big(\mathop{\sup}\limits_{\tau\in \mathcal{I}}\mathbb{E}\|\xi(\tau)\|^2\big)^{1/2}.$\\
%\end{align*}
Before proving the existence result, we set
\begin{align*}
\triangle_1 = \mathop{\max}\limits_{1\leq i \leq \mathcal{N}} \Big[&5\mathrm{M}^2\{8N_1 [N_2^\star]^2(L_{k_1}+L_{k_2})+3\lambda_i+4[N_2^\star]^2(3M_i+3M_b)\\&+a\lambda_f\mathop{\sup}\limits_{s\in\mathcal{I}}\mathfrak{n}(s)+2a^2\lambda_hTr(\mathtt{Q})\mathop{\sup}\limits_{s\in\mathcal{I}}\mathfrak{m}(s)\}+4[N_2^\star]^2(3M_i+5M_b)\Big],
\end{align*}and
\begin{align*}
\triangle_2 = \mathop{\max}\limits_{1\leq i\leq \mathcal{N}}\Big\{l_1^*+l_2^*+L_i+l_{m_i}^*+ \mathrm{M}(L_i+l_{m_i}^*+l_b^*)+l_b^*+4\mathrm{M}\mho+4\mathrm{M}\sqrt{a}\sqrt{Tr(\mathtt{Q})}\|\chi\|_{L^2} \Big\}.
\end{align*}
\begin{theorem}\label{thm3.1}
If $(S_\uppsi)-(S_7)$ hold, then the system (\ref{mainequation}) admits at least one solution on $(-\infty,a]$ provided that 
\begin{equation}\label{3.1}   
%\begin{center}
\max\{\triangle_1,\triangle_2\}\ <\ 1.
%\end{center}
\end{equation} 
\end{theorem} 
\begin{proof}
Define $\Gamma:\mathfrak{D}_a \rightarrow \mathfrak{D}_a$ by
\begin{equation}
(\Gamma \xi)(\tau)=
\begin{cases}
%\begin{aligned}
 \uppsi(\tau),\ \hspace{9.7cm} \tau \in \mathcal{I}_0; \\
 \mathcal{T}_q(\tau)[ \uppsi(0)- k_1(\xi) - b(0,\uppsi)]+\mathcal{S}_q(\tau)[\xi_1-k_2(\xi)]\\ + b(\tau,\xi_{\varrho(\tau,\xi_\tau)})+ \int_0^\tau \int_{\gamma} \mathcal{S}_q (\tau-s) f(s,\vartheta,\xi_{\varrho(s,\xi_s)}) \widetilde{Z}(ds,d\vartheta) \\+\int_0^\tau \mathcal{T}_q (\tau-s) (\int_{-\infty}^s h(z, \xi_{\varrho(z,\xi_{z})})d\omega(z))ds,\hspace{3.5cm} \tau \in [0,r_1];\\
 l_i(\xi_{r_i}) + m_i(\tau, \xi_{\varrho(\tau,\xi_\tau)}),\hspace{6.7cm} \tau\in \mathop{\cup}\limits_{i=1}^{\mathcal{N}}(r_i,s_i]; \\
 \mathcal{T}_q(\tau-s_i)[ l_i(\xi_{s_i}) + m_i(s_i, \xi_{\varrho(s_i,\xi_{s_i})}) - b(s_i,\xi_{\varrho(s_i,\xi_{s_i})})]\\+ b(\tau,\xi_{\varrho(\tau,\xi_\tau)})
+ \int_{s_i}^\tau \int_{\gamma} \mathcal{S}_q (\tau-s) f(s,\vartheta,\xi_{\varrho(s,\xi_s)}) \widetilde{Z}(ds,d\vartheta) \\+ \int_{s_i}^\tau \mathcal{T}_q (\tau-s) (\int_{-\infty}^s h(z, \xi_{\varrho(z,\xi_{z})})d\omega(z))ds,\hspace{3.3cm} \tau\in \mathop{\cup}\limits_{i=1}^{\mathcal{N}}(s_i,r_{i+1}].
%\end{aligned}
\end{cases}
\end{equation}
We shall show that $\Gamma$ has a fixed point in $\mathfrak{D}_a$ which is a mild solution of (\ref{mainequation}).\\
Let $\overline{\uppsi}(\cdot)\ :\ (-\infty,a]\ \rightarrow \mathcal{V}$ be given by  \begin{equation*}
\overline{\uppsi}(\tau)=
\begin{cases}
\uppsi(\tau) ,\hspace{1cm} \tau\in \mathcal{I}_0;\\
\mathcal{T}_q(\tau) \uppsi(0),\ \tau \in \mathcal{I}.
\end{cases}
\end{equation*}
Then $\overline{\uppsi}\in \mathfrak{D}_a$ and $\overline{\uppsi}_0=\uppsi$. Split  $\xi(\tau)=\overline{\uppsi}(\tau)+\zeta(\tau),\   -\infty<\tau\leq a$. Clearly, $\xi(\cdot)$ satisfies (\ref{mildsolution}) if and only if $\zeta_0=0$ and take $\bar{\zeta}$ characterized by
\begin{equation*}
\bar{\zeta}(\tau)=
\begin{cases}
0 ,  \qquad \tau\in \mathcal{I}_0;\\
\zeta(\tau),\quad  \tau \in \mathcal{I},
\end{cases}
\end{equation*}
where
\begin{equation}
\zeta(\tau)=
\begin{cases}
 \mathcal{T}_q(\tau)[-k_1(\overline{\uppsi}+\zeta) - b(0,\uppsi)]+\mathcal{S}_q(\tau)[\xi_1-k_2(\overline{\uppsi}+\zeta)] + b(\tau,\overline{\uppsi}_{\varrho(\tau,\overline{\uppsi}_\tau+\zeta_\tau)}+\zeta_{\varrho(\tau,\overline{\uppsi}_\tau+\zeta_\tau)}\\+ \int_0^\tau \int_{\gamma} \mathcal{S}_q (\tau-s) f(s,\vartheta,\overline{\uppsi}_{\varrho(s,\overline{\uppsi}_s+\zeta_s)}+\zeta_{\varrho(s,\overline{\uppsi}_s+\zeta_s)}) \widetilde{Z}(ds,d\vartheta) \\+\int_0^\tau \mathcal{T}_q (\tau-s) \big(\int_{-\infty}^s h(z, \overline{\uppsi}_{\varrho(z,\overline{\uppsi}_{z}+\zeta_{z})}+\zeta_{\varrho(z,\overline{\uppsi}_{z}+\zeta_{z})})d\omega(z)\big)ds,\ \hspace{1cm}  \tau \in [0,r_1];\\
 l_i(\overline{\uppsi}_{r_i}+\zeta_{r_i}) + m_i(\tau,\overline{\uppsi}_{\varrho(\tau,\overline{\uppsi}_\tau+\zeta_\tau)}+\zeta_{\varrho(\tau,\overline{\uppsi}_\tau+\zeta_\tau)})-\mathcal{T}_q(\tau)\uppsi(0),\ \hspace{1cm} \tau\in \mathop{\cup}\limits_{i=1}^{\mathcal{N}}(r_i,s_i]; \\
 \mathcal{T}_q(\tau-s_i)[ l_i(\overline{\uppsi}_{s_i}+\zeta_{s_i}) + m_i(s_i,\overline{\uppsi}_ {\varrho(s_i,\overline{\uppsi}_{s_i}+\zeta_{s_i})}+\zeta_{\varrho(s_i,\overline{\uppsi}_{s_i}+\zeta_{s_i})}) \\- b(s_i,\overline{\uppsi}_ {\varrho(s_i,\overline{\uppsi}_{s_i}+\zeta_{s_i})}+\zeta_{\varrho(s_i,\overline{\uppsi}_{s_i}+\zeta_{s_i})})]-\mathcal{T}_q(\tau)\uppsi(0)+ b(\tau,\overline{\uppsi}_{\varrho(\tau,\overline{\uppsi}_\tau+\zeta_\tau)}+\zeta_{\varrho(\tau,\overline{\uppsi}_\tau+\zeta_\tau)})\\
+ \int_{s_i}^\tau \int_{\gamma} \mathcal{S}_q (\tau-s) f(s,\vartheta,\overline{\uppsi}_{\varrho(s,\overline{\uppsi}_s+\zeta_s)}+\zeta_{\varrho(s,\overline{\uppsi}_s+\zeta_s)}) \widetilde{Z}(ds,d\vartheta) \\+ \int_{s_i}^\tau \mathcal{T}_q (\tau-s) (\int_{-\infty}^s h(z,\overline{\uppsi}_{\varrho(z,\overline{\uppsi}_{z}+\zeta_{z})}+\zeta_{\varrho(z,\overline{\uppsi}_{z}+\zeta_{z})})d\omega(z))ds,\ \hspace{1cm} \tau\in \mathop{\cup}\limits_{i=1}^{\mathcal{N}}(s_i,r_{i+1}].
\end{cases}
\end{equation}
Define $\mathfrak{D}_a^0=\{\zeta \in \mathfrak{D}_a:\ \zeta_0=0 \in \mathfrak{D}\}$, and for $\zeta \in \mathfrak{D}_a^0$, we set
\begin{align*}
\|\zeta\|_a= {\|\zeta_0 \|}_{\mathfrak{D}}+\big(\mathop{\sup}\limits_{s\in \mathcal{I}}\mathbb{E}\|\zeta(s)\|^2\ \big)^{1/2}= \big(\mathop{\sup}\limits_{s\in \mathcal{I}}\mathbb{E}\|\zeta(s)\|^2\ \big)^{1/2}.
\end{align*}
Then $(\mathfrak{D}_a^0,\|\cdot\|_a)$ is a Banach space. For every $\alpha>0$, let
%\begin{align*}
$\overline{\mathtt{B}}_\alpha=\{\zeta \in \mathfrak{D}_a^0:\ \|\zeta\|_a^2 \leq \alpha\}$.
%\end{align*} 
Clearly, $\overline{\mathtt{B}}_\alpha \subset \mathfrak{D}_a^0$ is convex, closed and bounded. Using Lemma~\ref{lemma3.1}, for $\zeta \in \overline{\mathtt{B}}_\alpha$, we have 
%\begin{equation}\label{3.4}
\begin{align}\label{3.4}
&\|\overline{\uppsi}_{\varrho(\tau,\overline{\uppsi}_\tau+\zeta_\tau)}+\zeta_{\varrho(\tau,\overline{\uppsi}_\tau+\zeta_\tau)}\|_{\mathfrak{D}}^2 \nonumber\\&\leq \ 2 \big(\|\overline{\uppsi}_{\varrho(\tau,\overline{\uppsi}_\tau+\zeta_\tau)}\|_{\mathfrak{D}}^2+\|\zeta_{\varrho(\tau,\overline{\uppsi}_\tau+\zeta_\tau)}\|_\mathfrak{D}^2 \big) \nonumber \\&\leq 4\Big\{[N_2^\star]^2 \mathop{\sup}\limits_{s\in \mathcal{I}}\mathbb{E}\|\overline{\uppsi}(s)\|^2+(N_3^\star+J^\star)^2 \| \overline{\uppsi}_0 \|_{\mathfrak{D}}^2 \nonumber +[N_2^\star]^2 \mathop{\sup}\limits_{s\in \mathcal{I}}\mathbb{E}\|\zeta(s)\|^2+(N_3^\star+J^\star)^2 \| \zeta_0 \|_\mathfrak{D}^2\Big\}\nonumber \\ &\leq 4\Big\{ [N_2^\star]^2 \mathrm{M}^2\|{\uppsi}(0)\|^2 + (N_3^\star+J^\star)^2\|\uppsi\|_\mathfrak{D}^2+[N_2^\star]^2 \mathop{\sup}\limits_{s\in \mathcal{I}}\mathbb{E}\|\zeta(s)\|^2\}\nonumber \\ &\leq 4\{[N_2^\star]^2 \alpha +[N_2^\star]^2 \mathrm{M}^2 N_1^2 \|{\uppsi}\|_\mathfrak{D}^2 + (N_3^\star+J^\star)^2\|\uppsi\|_\mathfrak{D}^2\} = 4[N_2^\star]^2 \alpha + C_0 = r^*\ ; \ \tau\in \mathcal{I},
\end{align}
where $ C_0 = \big([N_2^\star]^2 \mathrm{M}^2N_1^2 + (N_3^\star+J^\star)^2\big)\|\uppsi\|_\mathfrak{D}^2$.\\
Next, Lemma~\ref{lemma2.1} yields that
\begin{align}\label{3.5}
\|\overline{\uppsi}_\tau+\zeta_\tau\|_{\mathfrak{D}}^2 \  \leq&\  4\{[N_2^\star]^2 \mathop{\sup}\limits_{\tau\in \mathcal{I}}\mathbb{E}\|{\uppsi}(s)\|^2+[N_3^\star]^2 \|{\uppsi}(0) \|_{\mathfrak{D}}^2\nonumber \\&+[N_2^\star]^2 \mathop{\sup}\limits_{\tau\in \mathcal{I}}\mathbb{E}\|\zeta(s)\|^2+[N_3^\star]^2 \| \zeta_0 \|_\mathfrak{D}^2\}\nonumber \\ =&\  4[N_2^\star]^2 \alpha + C_1 = r^{**}\ , \ \tau\in \mathcal{I},
\end{align}
where $C_1 = \big([N_2^\star]^2 \mathrm{M}^2 N_1^2+ [N_3^\star]^2\big)\|\uppsi\|_\mathfrak{D}^2$.

Now define the operator $\Upsilon : \mathfrak{D}_a^0 \rightarrow \mathfrak{D}_a^0$ by
\begin{equation}
(\Upsilon \zeta)(\tau)=
\begin{cases}
%\begin{aligned}
 \mathcal{T}_q(\tau)[-k_1(\overline{\uppsi}+\zeta) - b(0,\uppsi)]+\mathcal{S}_q(\tau)[\xi_1-k_2(\overline{\uppsi}+\zeta)] + b(\tau,\overline{\uppsi}_{\varrho(\tau,\overline{\uppsi}_\tau+\zeta_\tau)}+\zeta_{\varrho(\tau,\overline{\uppsi}_\tau+\zeta_\tau)}) \\+ \int_0^\tau \int_{\gamma} \mathcal{S}_q (\tau-s) f(s,\vartheta,\overline{\uppsi}_{\varrho(s,\overline{\uppsi}_s+\zeta_s)}+\zeta_{\varrho(s,\overline{\uppsi}_s+\zeta_s)}) \widetilde{Z}(ds,d\vartheta) \\+\int_0^\tau \mathcal{T}_q (\tau-s) \big(\int_{-\infty}^s h(z, \overline{\uppsi}_{\varrho(z,\overline{\uppsi}_{z}+\zeta_{z})}+\zeta_{\varrho(z,\overline{\uppsi}_{z}+\zeta_{z})})d\omega(z)\big)ds,\ \hspace{3cm}  \tau \in [0,r_1];\\
 l_i(\overline{\uppsi}_{r_i}+\zeta_{r_i}) + m_i(\tau,\overline{\uppsi}_{\varrho(\tau,\overline{\uppsi}_\tau+\zeta_\tau)}+\zeta_{\varrho(\tau,\overline{\uppsi}_\tau+\zeta_\tau)})-\mathcal{T}_q(\tau)\uppsi(0),\ \hspace{2.7cm} \tau\in \mathop{\cup}\limits_{i=1}^{\mathcal{N}}(r_i,s_i]; \\
 \mathcal{T}_q(\tau-s_i)[ l_i(\overline{\uppsi}_{s_i}+\zeta_{s_i}) + m_i(s_i,\overline{\uppsi}_ {\varrho(s_i,\overline{\uppsi}_{s_i}+\zeta_{s_i})}+\zeta_{\varrho(s_i,\overline{\uppsi}_{s_i}+\zeta_{s_i})}) \\- b(s_i,\overline{\uppsi}_ {\varrho(s_i,\overline{\uppsi}_{s_i}+\zeta_{s_i})}+\zeta_{\varrho(s_i,\overline{\uppsi}_{s_i}+\zeta_{s_i})})]-\mathcal{T}_q(\tau)\uppsi(0)+ b(\tau,\overline{\uppsi}_{\varrho(\tau,\overline{\uppsi}_\tau+\zeta_\tau)}+\zeta_{\varrho(\tau,\overline{\uppsi}_\tau+\zeta_\tau)})\\
+ \int_{s_i}^\tau \int_{\gamma} \mathcal{S}_q (\tau-s) f(s,\vartheta,\overline{\uppsi}_{\varrho(s,\overline{\uppsi}_s+\zeta_s)}+\zeta_{\varrho(s,\overline{\uppsi}_s+\zeta_s)}) \widetilde{Z}(ds,d\vartheta) \\+ \int_{s_i}^\tau \mathcal{T}_q (\tau-s) \big(\int_{-\infty}^s h(z,\overline{\uppsi}_{\varrho(z,\overline{\uppsi}_{z}+\zeta_{z})}+\zeta_{\varrho(z,\overline{\uppsi}_{z}+\zeta_{z})})d\omega(z)\big)ds,\ \hspace{2cm} \tau\in \mathop{\cup}\limits_{i=1}^{\mathcal{N}}(s_i,r_{i+1}].
%\end{aligned}
\end{cases}
\end{equation}
Clearly, $\Gamma$ has a fixed point if and only if $\Upsilon$ has a fixed point. Therefore, it suffices to prove that $\Upsilon$ has a fixed point. For accessibility, the proof is splitted into three steps:\\ 
\textbf{Step \rom{1}:} We assert that $\Upsilon(\overline{\mathtt{B}}_\alpha) \subset \overline{\mathtt{B}}_\alpha$, for some $\alpha>0$.\\
On contrary, assume that $\Upsilon (\overline{\mathtt{B}}_\alpha) \not\subset \overline{\mathtt{B}}_\alpha$. Then for every $\alpha>0$, there is $\zeta^\alpha(\cdot) \in \overline{\mathtt{B}}_\alpha$ and $\tau=\tau(\alpha)\in \mathcal{I}$, such that $\Upsilon(\zeta^\alpha) \notin\ \overline{\mathtt{B}}_\alpha$, that is, $\alpha<\mathbb{E}\|(\Upsilon \zeta^\alpha)(\tau)\|^2$ for some $\tau=\tau(\alpha) \in \mathcal{I}$. In fact, for $\tau=\tau(\alpha)\in [0,r_1]$, we get
\begin{align}
\alpha\ &<  \ \mathbb{E}\|(\Upsilon \zeta^\alpha)(\tau)\|^2 \nonumber\\  \ &\leq \ 5 \Big[\mathbb{E}\|\mathcal{T}_q(\tau)[-k_1(\overline{\uppsi}+\zeta^\alpha) - b(0,\uppsi)]\|^2+\mathbb{E}\|\mathcal{S}_q(\tau)[\xi_1-k_2(\overline{\uppsi}+\zeta^\alpha)]\|^2 \nonumber \\&\quad\ \ +\mathbb{E} \|b(\tau,\overline{\uppsi}_{\varrho(\tau,\overline{\uppsi}_\tau+\zeta_\tau^\alpha)}+\zeta_{\varrho(\tau,\overline{\uppsi}_\tau+\zeta_\tau^\alpha)}^\alpha\|^2\nonumber \\&\quad\ \  + \mathbb{E}\|\int_0^\tau \int_{\gamma} \mathcal{S}_q (\tau-s) f(s,\vartheta,\overline{\uppsi}_{\varrho(s,\overline{\uppsi}_s+\zeta_s^\alpha)}+\zeta_{\varrho(s,\overline{\uppsi}_s+\zeta_s^\alpha)}^\alpha) \widetilde{Z}(ds,d\vartheta)\|^2 \nonumber \\&\quad\ \ +\mathbb{E}\|\int_0^\tau \mathcal{T}_q (\tau-s) \big(\int_{-\infty}^s h(z, \overline{\uppsi}_{\varrho(z,\overline{\uppsi}_{z}+\zeta_{z}^\alpha)}+\zeta_{\varrho(z,\overline{\uppsi}_{z}+\zeta_{z}^\alpha)}^\alpha)d\omega(z)\big)ds\|^2\Big]\nonumber \\&\quad =5\mathop{\sum}\limits_{i=1}^5 J_i.
\end{align}
Using assumptions $(S_1)$, $(S_2)$(i), $(S_3)$(i), and Estimates \ref{3.4} and \ref{3.5}, we obtain
%\begin{equation*}
\begin{align}
J_1 \leq& \ \mathbb{E}\|\mathcal{T}_q(\tau)[k_1(\overline{\uppsi}+\zeta^\alpha) + b(0,\uppsi)]\|^2\nonumber \\ \leq& \ 2\mathrm{M}^2[M_b(1+\|\uppsi\|_\mathfrak{D}^2)+L_{k_1}\big(1+N_1^2\|\overline{\uppsi}_\tau+\zeta_\tau^\alpha\|_\mathfrak{D}^2\big)]\nonumber \\ \leq& \ 2\mathrm{M}^2[M_b(1+\|\uppsi\|_\mathfrak{D}^2)+L_{k_1}\big(1+N_1^2(4[N_2^\star]^2\alpha+C_1)\big)],\\ 
J_2 \leq& \ \mathbb{E}\|\mathcal{S}_q(\tau)[\xi_1+k_2(\overline{\uppsi}+\zeta^\alpha)]\|^2 \nonumber \\ \leq& \ 2\mathrm{M}^2[\mathbb{E}\|\xi_1\|^2+L_{k_2}(1+N_1^2\|\overline{\uppsi}_\tau+\zeta_\tau^\alpha\|_\mathfrak{D}^2)]\nonumber \\ \leq& \ 2\mathrm{M}^2[\mathbb{E}\|\xi_1\|^2+L_{k_2}(1+N_1(4[N_2^\star]^2\alpha+C_1))],\\ 
J_3 \leq& \ \mathbb{E} \|b(\tau,\overline{\uppsi}_{\varrho(\tau,\overline{\uppsi}_\tau+\zeta_\tau^\alpha)}+\zeta_{\varrho(\tau,\overline{\uppsi}_\tau+\zeta_\tau^\alpha)}^\alpha)\|^2\nonumber \\ \leq& \ M_b (1+\|\overline{\uppsi}_{\varrho(\tau,\overline{\uppsi}_\tau+\zeta_\tau^\alpha)}+\zeta_{\varrho(\tau,\overline{\uppsi}_\tau+\zeta_\tau^\alpha)}^\alpha\|_\mathfrak{D}^2)\nonumber\\ \leq& \ M_b(1+4[N_2^\star]^2\alpha+C_0),
\end{align}
Using $(S_7)$(ii), Estimate \ref{3.4}, and Remark $3.3.3$ in \cite{Zhuthesis}, it follows that 
\begin{align}
J_4 \leq& \ \mathbb{E}\|\int_0^\tau \int_{\gamma} \mathcal{S}_q (\tau-s) f(s,\vartheta,\overline{\uppsi}_{\varrho(s,\overline{\uppsi}_s+\zeta_s^\alpha)}+\zeta_{\varrho(s,\overline{\uppsi}_s+\zeta_s^\alpha)}^\alpha) \widetilde{Z}(ds,d\vartheta)\|^2 \nonumber \\ \leq& \ \int_0^\tau \int_{\gamma} \mathbb{E}\|\mathcal{S}_q (\tau-s) f(s,\vartheta,\overline{\uppsi}_{\varrho(s,\overline{\uppsi}_s+\zeta_s^\alpha)}+\zeta_{\varrho(s,\overline{\uppsi}_s+\zeta_s^\alpha)}^\alpha)\|^2 \varkappa(d\vartheta)ds \nonumber \\ \leq& \ \mathrm{M}^2 \int_0^\tau \mathfrak{n}(s)\mathcal{A}_f(\|\overline{\uppsi}_{\varrho(s,\overline{\uppsi}_s+\zeta_s^\alpha)}+\zeta_{\varrho(s,\overline{\uppsi}_s+\zeta_s^\alpha)}^\alpha\|_\mathfrak{D}^2)ds \leq\ \mathrm{M}^2r_1 \mathcal{A}_f(r^*)\mathop{\sup}\limits_{s\in[0,r_1]}\mathfrak{n}(s),
\end{align}
Next, $(S_6)$(ii), (iii), Lemma \ref{lemma2.3}, Holder's inequality, and Estimate \ref{3.4} yield that
\begin{align}
J_5 &\leq \ \int_0^\tau \|\mathcal{T}_q (\tau-s)\|^2ds \int_0^\tau\mathbb{E}\|\Big(\int_{-\infty}^0 h(z, \overline{\uppsi}_{\varrho(z,\overline{\uppsi}_{z}+\zeta_{z}^\alpha)}+\zeta_{\varrho(z,\overline{\uppsi}_{z}+\zeta_{z}^\alpha)}^\alpha)d\omega(z)\nonumber \\&\quad\ +\int_{0}^s h(z, \overline{\uppsi}_{\varrho(z,\overline{\uppsi}_{z}+\zeta_{z}^\alpha)}+\zeta_{\varrho(z,\overline{\uppsi}_{z}+\zeta_{z}^\alpha)}^\alpha)d\omega(z)\Big)\|^2ds\nonumber \\ &\leq \ \mathrm{M}^2 r_1 \int_0^\tau [2 M_h + 2  \mathbb{E}\int_0^s \|h(z, \overline{\uppsi}_{\varrho(z,\overline{\uppsi}_{z}+\zeta_{z}^\alpha)}+\zeta_{\varrho(z,\overline{\uppsi}_{z}+\zeta_{z}^\alpha)}^\alpha) d\omega(z)\|^2 ]ds\nonumber \\&\leq \mathrm{M}^2r_1[2M_hr_1+2Tr(\mathtt{Q})\int_0^\tau\mathbb{E}\int_0^s \|h(z, \overline{\uppsi}_{\varrho(z,\overline{\uppsi}_{z}+\zeta_{z}^\alpha)}+\zeta_{\varrho(z,\overline{\uppsi}_{z}+\zeta_{z}^\alpha)}^\alpha) dz\|_\mathtt{Q}^2ds] \nonumber \\ &\leq \ \mathrm{M}^2r_1^2[2M_h+2Tr(\mathtt{Q}) \mathcal{A}_h(\|\overline{\uppsi}_{\varrho(z,\overline{\uppsi}_{z}+\zeta_{z}^\alpha)}+\zeta_{\varrho(z,\overline{\uppsi}_{z}+\zeta_{z}^\alpha)}^\alpha\|_\mathfrak{D}^2)\mathop{\sup}\limits_{s\in[0,r_1]}\mathfrak{m}(s)]\nonumber \\ &\leq \ \mathrm{M}^2r_1^2\big[2M_h+2Tr(\mathtt{Q}) \mathcal{A}_h(r^*)\mathop{\sup}\limits_{s\in[0,r_1]}\mathfrak{m}(s)\big].
\end{align}
On combining (3.7) to (3.12), for $\tau=\tau(\alpha)\in[0,r_1]$, we get
\begin{align*}
\alpha <\ \mathbb{E}&\|(\Upsilon \zeta^\alpha)(\tau)\|^2 \\ \leq  5& \Big[2\mathrm{M}^2\{M_b(1+\|\uppsi\|_\mathfrak{D}^2)+L_{k_1}\big(1+N_1^2(4[N_2^\star]^2\alpha+C_1)\big)+\mathbb{E}\|\xi_1\|^2+L_{k_2}\big(1+N_1^2(4[N_2^\star]^2\alpha+C_1)\big)\}\\&+ M_b(1+4[N_2^\star]^2\alpha+C_0)+ \mathrm{M}^2r_1 \mathcal{A}_f(r^*)\mathop{\sup}\limits_{s\in[0,r_1]}\mathfrak{n}(s)+ \mathrm{M}_2r_1^2[2M_h+2Tr(\mathtt{Q}) \mathcal{A}_h(r^*)\mathop{\sup}\limits_{s\in[0,r_1]}\mathfrak{m}(s)]\Big].
\end{align*}
Further, for $\tau=\tau(\alpha)\in \mathop{\cup}\limits_{i=1}^{\mathcal{N}}(r_i,s_i]$, using $(S_4)$(ii) and $(S_5)$(i), we obtain
\begin{align*}
\alpha< \mathbb{E}\|(\Upsilon \zeta^\alpha)(\tau)\|^2 \leq& \ 3[\mathbb{E}\|l_i(\overline{\uppsi}_{r_i}+\zeta_{r_i}^\alpha)\|^2 + \mathbb{E}\|m_i(\tau,\overline{\uppsi}_{\varrho(\tau,\overline{\uppsi}_\tau+\zeta_\tau^\alpha)}+\zeta_{\varrho(\tau,\overline{\uppsi}_\tau+\zeta_\tau^\alpha)}^\alpha)\|^2+\mathbb{E}\|\mathcal{T}_q(\tau)\uppsi(0)\|^2]\\
\leq& \ 3[\mathcal{A}_{l_i}(r^{**})+M_i(1+C_0+4[N_2^\star]^2\alpha)+ \mathrm{M}^2N_1^2\|\uppsi\|_\mathfrak{D}^2].
\end{align*}
Lastly, for $\tau=\tau(\alpha)\in \mathop{\cup}\limits_{i=1}^{\mathcal{N}}(s_i,r_{i+1}]$, a set of similar arguments as above imply that
\begin{align*}
 \alpha\ &< \ \mathbb{E}\|(\Upsilon \zeta^\alpha)(\tau)\|^2 
% \\ &\leq \  5 \Big[\mathbb{E}\|\mathcal{T}_q(\tau-s_i)[ l_i(\overline{\uppsi}_{s_i}+\zeta_{s_i}^\alpha) + m_i(s_i,\overline{\uppsi}_ {\varrho(s_i,\overline{\uppsi}_{s_i}+\zeta_{s_i}^\alpha)}+\zeta_{\varrho(s_i,\overline{\uppsi}_{s_i}+\zeta_{s_i}^\alpha)}^\alpha)\\&\quad- b(s_i,\overline{\uppsi}_ {\varrho(s_i,\overline{\uppsi}_{s_i}+\zeta_{s_i}^\alpha)}+\zeta_{\varrho(s_i,\overline{\uppsi}_{s_i}+\zeta_{s_i}^\alpha)}^\alpha)]\|^2+\mathbb{E}\| b(\tau,\overline{\uppsi}_{\varrho(\tau,\overline{\uppsi}_\tau+\zeta_\tau^\alpha)}+\zeta_{\varrho(\tau,\overline{\uppsi}_\tau+\zeta_\tau^\alpha)})\|^2\\&\quad+\mathbb{E}\|\mathcal{T}_q(\tau)\uppsi(0)\|^2+ \mathbb{E}\| \int_{s_i}^\tau \int_{\gamma} \mathcal{S}_q (\tau-s) f(s,\vartheta,\overline{\uppsi}_{\varrho(s,\overline{\uppsi}_s+\zeta_s^\alpha)}+\zeta_{\varrho(s,\overline{\uppsi}_s+\zeta_s^\alpha)}^\alpha) \widetilde{Z}(ds,d\vartheta)\|^2 \\&\quad+ \mathbb{E}\| \int_{s_i}^\tau \mathcal{T}_q (\tau-s) \big(\int_{-\infty}^s h(z,\overline{\uppsi}_{\varrho(z,\overline{\uppsi}_{z}+\zeta_{z}^\alpha)}+\zeta_{\varrho(z,\overline{\uppsi}_{z}+\zeta_{z}^\alpha)}^\alpha)d\omega(z)\big)ds\|^2\Big]
 \\ \ &\leq \  15\mathrm{M}^2\big[\mathcal{A}_{l_i}(r^{**})+M_i(1+C_0+4[N_2^\star]^2\alpha)+M_b(1+C_0+4[N_2^\star]^2\alpha)\big]+5\mathrm{M}^2N_1^2\|\uppsi\|_\mathfrak{D}^2\\&\quad\ +5M_b(1+C_0+4[N_2^\star]^2\alpha)+\mathrm{M}^2(r_{i+1}-s_i) \mathcal{A}_f(r^*)\mathop{\sup}\limits_{s\in (s_i,r_{i+1}]}\mathfrak{n}(s)\\&\quad\ +5\mathrm{M}^2(r_{i+1}-s_i)^2\big[2M_h+2Tr(\mathtt{Q})\mathcal{A}_h(r^*)\mathop{\sup}\limits_{s\in (s_i,r_{i+1}]}\mathfrak{m}(s)\big].
\end{align*}
Thus for $\tau=\tau(\alpha)\in \mathcal{I}$,
\begin{align}\label{inequality}
\alpha\ <\  \mathbb{E}\|(\Upsilon \zeta^\alpha)(\tau)\|^2 &\leq\ \widehat{M}+5\mathrm{M}^2\big\{8N_1 [N_2^\star]^2\alpha(L_{k_1}+L_{k_2})+3\mathcal{A}_{l_i}(r^{**})+4[N_2^\star]^2\alpha(3M_i+3M_b)\nonumber\\&\quad+a\mathcal{A}_f(r^*)\mathop{\sup}\limits_{s\in\mathcal{I}}\mathfrak{n}(s)+2a^2Tr(\mathtt{Q})\mathcal{A}_h(r^*)\mathop{\sup}\limits_{s\in\mathcal{I}}\mathfrak{m}(s)\big\}+4[N_2^\star]^2\alpha(3M_i+5M_b),
\end{align}
\begin{align*}
\text{where} \ \widehat{M}&= \ 5\mathrm{M}^2\{2M_b(1+\|\uppsi\|_\mathfrak{D}^2)+L_{k_1}(1+N_1^2C_1)+\|\xi\|^2+L_{k_2}(1+N_1^2C_1)+\|\uppsi\|_\mathfrak{D}^2\\&\qquad+(3M_b+3M_i)(C_0+1)\}+3M_i(C_0+1).
\end{align*}
 Dividing both sides of Inequality \ref{inequality} by $\alpha$ and letting as $\alpha \rightarrow \infty$, we get
\begin{align*}
1<\ \mathop{\max}\limits_{1\leq i \leq \mathcal{N}} &[5\mathrm{M}^2\{8N_1 [N_2^\star]^2(L_{k_1}+L_{k_2})+3\lambda_i+4[N_2^\star]^2(3M_i+3M_b )\\&+a\lambda_f \mathop{\sup}\limits_{s\in\mathcal{I}}\mathfrak{n}(s)+2a^2\lambda_hTr(\mathtt{Q})\mathop{\sup}\limits_{s\in\mathcal{I}}\mathfrak{m}(s)\}+4[N_2^\star]^2(3M_i+5M_b)],
\end{align*}
which is a contradiction to (\ref{3.1}). Therefore, there is some $\alpha>0$ to ensure that $\Upsilon(\overline{\mathtt{B}}_\alpha) \subset \overline{\mathtt{B}}_\alpha$.\\
\textbf{Step \rom{2}:}  $\Upsilon : \overline{\mathtt{B}}_\alpha  \rightarrow  \overline{\mathtt{B}}_\alpha$ is continuous.\\
Let $\{\zeta^{(n)}\}_{n\in \mathbb{N}}\ \subseteq \ \mathfrak{D}_a^0$ be a sequence such that $\zeta^{(n)} \rightarrow  \zeta \in \mathfrak{D}_a^0  $ as $(n \rightarrow \infty)$. So there is an $\alpha>0$ such that $\mathbb{E}\|\zeta^{(n)}(\tau)\|^2 \leq \alpha$ for all $n$ and a.s.  $\tau \in \mathcal{I}$, $\zeta^{(n)} \in~ \overline{\mathtt{B}}_\alpha$ and $\zeta \in \overline{\mathtt{B}}_\alpha$.
 Clearly, Estimates (\ref{3.4}) and (\ref{3.5}) hold for each $\zeta^{(n)}$.
%\begin{align*}
%\|\overline{\uppsi}_{\varrho(\tau,\overline{\uppsi}_\tau+\zeta_\tau^{(n)})}+\zeta_{\varrho(\tau,\overline{\uppsi}_\tau+\zeta_\tau^{(n)})}^{(n)}\|_\mathfrak{D}^2 \leq \ r^*\\ \\
%\|\overline{\uppsi}_\tau+\zeta_\tau^{(n)}\|_\mathfrak{D}^2 \leq \ r^{**}
%\end{align*}
% Next, Lemma \ref{lemma2.2}, yields that
%\begin{align*}
%\|\zeta^{(n)}_{\varrho(\tau,\overline{\uppsi}_\tau+\zeta_\tau)}-\zeta_{\varrho(\tau,\overline{\uppsi}_\tau+\zeta_\tau)}\|_\mathfrak{D}^2 \ \leq& \ 2[(N_2(\tau))^2\mathop{\sup}\limits_{s \in [0,t]}\|\zeta^{(n)}(s)-\zeta (s)\|_\mathfrak{D}^2+(N_3(t))^2\mathop{\sup}\limits_{s \in \mathcal{I}}\|\zeta^{(n)}(0)-\zeta_0\|_\mathfrak{D}^2]\\ =&\ 2(N_2(t))^2\mathop{\sup}\limits_{s \in [0,t]}\{\|\zeta^{(n)}(s)-\zeta(s)\|_\mathfrak{D}^2\}\\ \leq&\ 2[N_2^\star]^2 \|\zeta^{(n)}-\zeta\|_{\mathfrak{D}_a}^2 \rightarrow 0,\ as\ n \rightarrow \infty.
%\end{align*}
By Axiom \rom{1}, we observe that $\zeta_\tau^{(n)}\rightarrow \zeta_\tau$ as $n \rightarrow \infty$ uniformly for $\tau\in(-\infty,a]$. 
Now $(S_2)$(i), $(S_6)$(i), $(S_7)$(i) and the inequality
\begin{align*}
&\|h(z, \overline{\uppsi}_{\varrho(z,\overline{\uppsi}_{z}+\zeta_{z}^{(n)})}+\zeta_{\varrho(z,\overline{\uppsi}_{z}+\zeta_{z}^{(n)})}^{(n)})- h(z, \overline{\uppsi}_{\varrho(z,\overline{\uppsi}_{z}+\zeta_{z})}+\zeta_{\varrho(z,\overline{\uppsi}_{z}+\zeta_{z)}})\|^2\\&\leq 2\|h(z, \overline{\uppsi}_{\varrho(z,\overline{\uppsi}_{z}+\zeta_{z}^{(n)})}+\zeta_{\varrho(z,\overline{\uppsi}_{z}+\zeta_{z}^{(n)})}^{(n)})- h(z, \overline{\uppsi}_{\varrho(z,\overline{\uppsi}_{z}+\zeta_{z}^{(n)})}+\zeta_{\varrho(z,\overline{\uppsi}_{z}+\zeta_{z}^{(n)})})\|^2\\&\quad+2\| h(z, \overline{\uppsi}_{\varrho(z,\overline{\uppsi}_{z}+\zeta_{z}^{(n)})}+\zeta_{\varrho(z,\overline{\uppsi}_{z}+\zeta_{z}^{(n)})})- h(z, \overline{\uppsi}_{\varrho(z,\overline{\uppsi}_{z}+\zeta_{z})}+\zeta_{\varrho(z,\overline{\uppsi}_{z}+\zeta_{z})})\|^2,
\end{align*} yield that $h(z, \overline{\uppsi}_{\varrho(z,\overline{\uppsi}_{z}+\zeta_{z}^{(n)})}+\zeta_{\varrho(z,\overline{\uppsi}_{z}+\zeta_{z}^{(n)})}^{(n)}) \rightarrow h(z, \overline{\uppsi}_{\varrho(z,\overline{\uppsi}_{z}+\zeta_{z})}+\zeta_{\varrho(z,\overline{\uppsi}_{z}+\zeta_{z})}),\ n \rightarrow \infty$.\\
Similarly,
\begin{align*}
%h(z, \overline{\uppsi}_{\varrho(z,\overline{\uppsi}_{z}+\zeta_{z}^{(n)})}+\zeta_{\varrho(z,\overline{\uppsi}_{z}+\zeta_{z}^{(n)})}^{(n)}) \rightarrow h(z, \overline{\uppsi}_{\varrho(z,\overline{\uppsi}_{z}+\zeta_{z})}+\zeta_{\varrho(z,\overline{\uppsi}_{z}+\zeta_{z})}),\ n \rightarrow \infty, \\
 b(\tau,\overline{\uppsi}_{\varrho(\tau,\overline{\uppsi}_\tau+\zeta_\tau^{(n)})}+\zeta_{\varrho(\tau,\overline{\uppsi}_\tau+\zeta_\tau^{(n)})}^{(n)}) \rightarrow b(\tau,\overline{\uppsi}_{\varrho(\tau,\overline{\uppsi}_\tau+\zeta_\tau)}+\zeta_{\varrho(\tau,\overline{\uppsi}_\tau+\zeta_\tau)}),\ n\rightarrow \infty,\\
f(s,\vartheta,\overline{\uppsi}_{\varrho(s,\overline{\uppsi}_s+\zeta_s^{(n)})}+\zeta_{\varrho(s,\overline{\uppsi}_s+\zeta_s^{(n)})}^{(n)})\rightarrow f(s,\vartheta,\overline{\uppsi}_{\varrho(s,\overline{\uppsi}_s+\zeta_s)}+\zeta_{\varrho(s,\overline{\uppsi}_s+\zeta_s)}),\ n \rightarrow \infty.
\end{align*}
Moreover, \begin{center}
$k_1(\overline{\uppsi}+\zeta^{(n)})\rightarrow k_1(\overline{\uppsi}+\zeta),\  k_2(\overline{\uppsi}+\zeta^{(n)})\rightarrow k_2(\overline{\uppsi}+\zeta),$ $n\rightarrow \infty,$
\end{center}
\begin{align*}
\int_0^s\mathbb{E}\|h(z, \overline{\uppsi}_{\varrho(z,\overline{\uppsi}_{z}+\zeta_{z}^{(n)})}+\zeta_{\varrho(z,\overline{\uppsi}_{z}+\zeta_{z}^{(n)})}^{(n)}) - h(z, \overline{\uppsi}_{\varrho(z,\overline{\uppsi}_{z}+\zeta_{z)}}+\zeta_{\varrho(z,\overline{\uppsi}_{z}+\zeta_{z)}})\|^2dz \leq 2\mathfrak{m}(s)\mathcal{A}_h(r^*),\\
\int_\gamma\big[\mathbb{E}\|f(s,\vartheta,\overline{\uppsi}_{\varrho(s,\overline{\uppsi}_s+\zeta_s^{(n)})}+\zeta_{\varrho(s,\overline{\uppsi}_s+\zeta_s^{(n)})}^{(n)})- f(s,\vartheta,\overline{\uppsi}_{\varrho(s,\overline{\uppsi}_s+\zeta_s)}+\zeta_{\varrho(s,\overline{\uppsi}_s+\zeta_s)})\|^2\varkappa d(\vartheta) \leq 2\mathfrak{n}(s)\mathcal{A}_f(r^{*}).
\end{align*}
Thus for $\tau \in[0,r_1]$, by the Lebesgue dominated convergence theorem, we have
%$\mathbb{E}\|(\Upsilon \zeta^{(n)})(\tau)-(\Upsilon \zeta)(\tau)\|^2$
\begin{align*}
&\mathbb{E}\|(\Upsilon \zeta^{(n)})(\tau)-(\Upsilon \zeta)(\tau)\|^2\\ &\leq  5\mathrm{M}^2 \Big[\mathbb{E}\|k_1(\overline{\uppsi}+\zeta^{(n)})-k_1(\overline{\uppsi}+\zeta)\|^2+\mathbb{E}\|k_2(\overline{\uppsi}+\zeta^{(n)})-k_2(\overline{\uppsi}+\zeta)\|^2] \\&\quad+\mathbb{E} \|b(\tau,\overline{\uppsi}_{\varrho(\tau,\overline{\uppsi}_\tau+\zeta_\tau^{(n)})}+\zeta_{\varrho(\tau,\overline{\uppsi}_\tau+\zeta_\tau^{(n)})}^{(n)}) -b(\tau,\overline{\uppsi}_{\varrho(\tau,\overline{\uppsi}_\tau+\zeta_\tau)}+\zeta_{\varrho(\tau,\overline{\uppsi}_\tau+\zeta_\tau)})\|^2\\&\quad+ \int_0^\tau \int_{\gamma} \mathbb{E}\| f(s,\vartheta,\overline{\uppsi}_{\varrho(s,\overline{\uppsi}_s+\zeta_s^{(n)})}+\zeta_{\varrho(s,\overline{\uppsi}_s+\zeta_s^{(n)})}^{(n)})- f(s,\vartheta,\overline{\uppsi}_{\varrho(s,\overline{\uppsi}_s+\zeta_s)}+\zeta_{\varrho(s,\overline{\uppsi}_s+\zeta_s)})\|^2 \varkappa(d\vartheta)ds \\&\quad+\int_0^\tau \mathbb{E}\| \int_0^s \Big[h(z, \overline{\uppsi}_{\varrho(z,\overline{\uppsi}_{z}+\zeta_{z}^{(n)})}+\zeta_{\varrho(z,\overline{\uppsi}_{z}+\zeta_{z}^{(n)})}^{(n)})-h(z, \overline{\uppsi}_{\varrho(z,\overline{\uppsi}_{z}+\zeta_{z})}+\zeta_{\varrho(z,\overline{\uppsi}_{z}+\zeta_{z})})\Big]d\omega(z)\|^2ds\\ &\leq  5\mathrm{M}^2 \Big[\mathbb{E}\|k_1(\overline{\uppsi}+\zeta^{(n)})-k_1(\overline{\uppsi}+\zeta)\|^2+\mathbb{E}\|k_2(\overline{\uppsi}+\zeta^{(n)})-k_2(\overline{\uppsi}+\zeta)\|^2] \\&\quad+\mathbb{E} \|b(\tau,\overline{\uppsi}_{\varrho(\tau,\overline{\uppsi}_\tau+\zeta_\tau^{(n)})}+\zeta_{\varrho(\tau,\overline{\uppsi}_\tau+\zeta_\tau^{(n)})}^{(n)}) -b(\tau,\overline{\uppsi}_{\varrho(\tau,\overline{\uppsi}_\tau+\zeta_\tau)}+\zeta_{\varrho(\tau,\overline{\uppsi}_\tau+\zeta_\tau)})\|^2\\&\quad+\int_0^\tau \int_{\gamma} \mathbb{E}\| f(s,\vartheta,\overline{\uppsi}_{\varrho(s,\overline{\uppsi}_s+\zeta_s^{(n)})}+\zeta_{\varrho(s,\overline{\uppsi}_s+\zeta_s^{(n)})}^{(n)})- f(s,\vartheta,\overline{\uppsi}_{\varrho(s,\overline{\uppsi}_s+\zeta_s)}+\zeta_{\varrho(s,\overline{\uppsi}_s+\zeta_s)})\|^2 \varkappa(d\vartheta)ds\\&\quad+Tr(\mathtt{Q})\int_0^\tau  \int_0^s\mathbb{E}\|h(z, \overline{\uppsi}_{\varrho(z,\overline{\uppsi}_{z}+\zeta_{z}^{(n)})}+\zeta_{\varrho(z,\overline{\uppsi}_{z}+\zeta_{z}^{(n)})}^{(n)})-h(z, \overline{\uppsi}_{\varrho(z,\overline{\uppsi}_{z}+\zeta_{z})}+\zeta_{\varrho(z,\overline{\uppsi}_{z}+\zeta_{z})})\|^2dz ds\\ &\rightarrow\  0 \ \text{as} \ n \rightarrow \infty.
\end{align*}
Next, for $\tau\in \mathop{\cup}\limits_{i=1}^\mathcal{N} (s_i,r_{i+1}]$, using $(S_4)$ and $(S_5)$, we obtain
\begin{align*}
\mathbb{E}\|(\Upsilon \zeta^{(n)})(\tau)-(\Upsilon \zeta)(\tau)\|^2  \leq \ 2&\mathbb{E}\|l_i(\overline{\uppsi}_{r_i}+\zeta_{r_i}^{(n)})-l_i(\overline{\uppsi}_{r_i}+\zeta_{r_i})\|^2 \\&+ 2\mathbb{E}\|m_i(\tau,\overline{\uppsi}_{\varrho(\tau,\overline{\uppsi}_\tau+\zeta_\tau^{(n)})}+\zeta_{\varrho(\tau,\overline{\uppsi}_\tau+\zeta_\tau^{(n)})}^{(n)})-m_i(\tau,\overline{\uppsi}_{\varrho(\tau,\overline{\uppsi}_\tau+\zeta_\tau)}+\zeta_{\varrho(\tau,\overline{\uppsi}_\tau+\zeta_\tau)})\|^2\\ &\rightarrow \ 0 \ \text{as} \ n \rightarrow\infty. 
\end{align*}
Similarly, for $\tau\in \mathop{\cup}\limits_{i=1}^\mathcal{N} (r_i,s_i],$ we get
\begin{align*}
 &\mathbb{E}\|(\Upsilon \zeta^{(n)})(\tau)-(\Upsilon \zeta)(\tau)\|^2\\&\leq 4\mathrm{M}^2\Big[3\mathbb{E}\|l_i(\overline{\uppsi}_{s_i}+\zeta_{s_i}^{(n)})-l_i(\overline{\uppsi}_{s_i}+\zeta_{s_i})\|^2 \\&\quad+3\mathbb{E} \|m_i(s_i,\overline{\uppsi}_ {\varrho(s_i,\overline{\uppsi}_{s_i}+\zeta_{s_i}^{(n)})}+\zeta_{\varrho(s_i,\overline{\uppsi}_{s_i}+\zeta_{s_i}^{(n)})}^{(n)})- m_i(s_i,\overline{\uppsi}_ {\varrho(s_i,\overline{\uppsi}_{s_i}+\zeta_{s_i})}+\zeta_{\varrho(s_i,\overline{\uppsi}_{s_i}+\zeta_{s_i})})\|^2 \\&\quad+3\mathbb{E}\| b(s_i,\overline{\uppsi}_ {\varrho(s_i,\overline{\uppsi}_{s_i}+\zeta_{s_i}^{(n)})}+\zeta_{\varrho(s_i,\overline{\uppsi}_{s_i}+\zeta_{s_i}^{(n)})}^{(n)})-b(s_i,\overline{\uppsi}_ {\varrho(s_i,\overline{\uppsi}_{s_i}+\zeta_{s_i})}+\zeta_{\varrho(s_i,\overline{\uppsi}_{s_i}+\zeta_{s_i})})\|^2\\&\quad+\mathbb{E}\| b(\tau,\overline{\uppsi}_{\varrho(\tau,\overline{\uppsi}_\tau+\zeta_\tau^n)}+\zeta_{\varrho(\tau,\overline{\uppsi}_\tau+\zeta_\tau^n)}^n)-b(\tau,\overline{\uppsi}_{\varrho(\tau,\overline{\uppsi}_\tau+\zeta_\tau)}+\zeta_{\varrho(\tau,\overline{\uppsi}_\tau+\zeta_\tau)})\|^2 \\&\quad+ \int_{s_i}^\tau \int_{\gamma} \mathbb{E}\| (f(s,\vartheta,\overline{\uppsi}_{\varrho(s,\overline{\uppsi}_s+\zeta_s^{(n)})}+\zeta_{\varrho(s,\overline{\uppsi}_s+\zeta_s^{(n)})}^{(n)})-f(s,\vartheta,\overline{\uppsi}_{\varrho(s,\overline{\uppsi}_s+\zeta_s)}+\zeta_{\varrho(s,\overline{\uppsi}_s+\zeta_s)}))\|^2 \varkappa(d\vartheta)ds \\&\quad+ \int_{s_i}^\tau \mathbb{E}\|\int_0^s \big[h(z,\overline{\uppsi}_{\varrho(z,\overline{\uppsi}_{z}+\zeta_{z}^{(n)})}+\zeta_{\varrho(z,\overline{\uppsi}_{z}+\zeta_{z}^{(n)})}^{(n)})-h(z,\overline{\uppsi}_{\varrho(z,\overline{\uppsi}_{z}+\zeta_{z})}+\zeta_{\varrho(z,\overline{\uppsi}_{z}+\zeta_{z})})\big]d\omega(z)\|^2ds\Big]\\&\leq 4\mathrm{M}^2\Big[3\mathbb{E}\|l_i(\overline{\uppsi}_{s_i}+\zeta_{s_i}^{(n)})-l_i(\overline{\uppsi}_{s_i}+\zeta_{s_i})\|^2 \\&\quad+3\mathbb{E} \|m_i(s_i,\overline{\uppsi}_ {\varrho(s_i,\overline{\uppsi}_{s_i}+\zeta_{s_i}^{(n)})}+\zeta_{\varrho(s_i,\overline{\uppsi}_{s_i}+\zeta_{s_i}^{(n)})}^{(n)})- m_i(s_i,\overline{\uppsi}_ {\varrho(s_i,\overline{\uppsi}_{s_i}+\zeta_{s_i})}+\zeta_{\varrho(s_i,\overline{\uppsi}_{s_i}+\zeta_{s_i})})\|^2 \\&\quad+3\mathbb{E}\| b(s_i,\overline{\uppsi}_ {\varrho(s_i,\overline{\uppsi}_{s_i}+\zeta_{s_i}^{(n)})}+\zeta_{\varrho(s_i,\overline{\uppsi}_{s_i}+\zeta_{s_i}^{(n)})}^{(n)})-b(s_i,\overline{\uppsi}_ {\varrho(s_i,\overline{\uppsi}_{s_i}+\zeta_{s_i})}+\zeta_{\varrho(s_i,\overline{\uppsi}_{s_i}+\zeta_{s_i})})\|^2\\&\quad+\mathbb{E}\| b(\tau,\overline{\uppsi}_{\varrho(\tau,\overline{\uppsi}_\tau+\zeta_\tau^n)}+\zeta_{\varrho(\tau,\overline{\uppsi}_\tau+\zeta_\tau^n)}^n)-b(\tau,\overline{\uppsi}_{\varrho(\tau,\overline{\uppsi}_\tau+\zeta_\tau)}+\zeta_{\varrho(\tau,\overline{\uppsi}_\tau+\zeta_\tau)})\|^2\\&\quad+ \int_{s_i}^\tau \int_{\gamma} \mathbb{E}\| (f(s,\vartheta,\overline{\uppsi}_{\varrho(s,\overline{\uppsi}_s+\zeta_s^{(n)})}+\zeta_{\varrho(s,\overline{\uppsi}_s+\zeta_s^{(n)})}^{(n)})-f(s,\vartheta,\overline{\uppsi}_{\varrho(s,\overline{\uppsi}_s+\zeta_s)}+\zeta_{\varrho(s,\overline{\uppsi}_s+\zeta_s)}))\|^2 \varkappa(d\vartheta)ds \\&\quad+Tr(\mathtt{Q}) \int_{s_i}^\tau \int_0^s \mathbb{E}\|h(z,\overline{\uppsi}_{\varrho(z,\overline{\uppsi}_{z}+\zeta_{z}^{(n)})}+\zeta_{\varrho(z,\overline{\uppsi}_{z}+\zeta_{z}^{(n)})}^{(n)})-h(z,\overline{\uppsi}_{\varrho(z,\overline{\uppsi}_{z}+\zeta_{z})}+\zeta_{\varrho(z,\overline{\uppsi}_{z}+\zeta_{z})})\|^2dz ds\Big]\\ &\   \rightarrow \ 0 \ as \ n \rightarrow \infty.
\end{align*}
Hence $\Upsilon: \overline{\mathtt{B}}_\alpha \ \rightarrow \ \overline{\mathtt{B}}_\alpha$ is continuous.\\ 
\textbf{Step \rom{3}:} The Mönch condition holds. \par Let $\mathbf{G}\subset \overline{\mathtt{B}}_\alpha$ be countable and $\mathbf{G}\subset \overline{\textbf{co}}(\{0\}\cup \Upsilon(\mathbf{G}))$. Without loss of generality, consider  $\mathbf{G}=\{\zeta^n\}_{n=1}^\infty$. We assert that $\{\Upsilon (\zeta^n)\}_{n=1}^\infty$ is equicontinuous on $\mathcal{I}$, then clearly $\mathbf{G}\subset \overline{\textbf{co}}(\{0\}\cup \Upsilon(\mathbf{G}))$ is  equicontinuous on $\mathcal{I}$. Obviously $\{\Upsilon (\zeta^n)\}_{n=1}^\infty, \ \zeta^n \in \overline{\mathtt{B}}_\alpha$ is equicontinuous at $\tau=0$.
By $(S_1),\ \mathcal{T}_q(\cdot)$ and $\mathcal{S}_q(\cdot)$ are strongly continuous, the map $\tau\mapsto \mathcal{T}_q(\tau- \cdot)$ and $\tau\rightarrow \mathcal{S}_q(\tau- \cdot)$ are continuous in the uniform operator topology on $(0,a]$. 
Let $s_i<\epsilon< \tau \leq r_{i+1},\ i=0,1,\cdots,\mathcal{N}$ and $\delta>0$ such that for $\eta_1,\eta_2 \in \mathop{\cup}\limits_{i=0}^\mathcal{N} (s_i,r_{i+1}]$ with $|\eta_1-\eta_2|< \delta$,
\begin{align*}
\max\big\{\|\mathcal{T}_q(\eta_1)-\mathcal{T}_q(\eta_2)\|,\  \|\mathcal{S}_q(\eta_1)-\mathcal{S}_q(\eta_2)\|\big\} \ <\ \epsilon.
\end{align*}
For every $\zeta^n \in \overline{\mathtt{B}}_\alpha,\ 0<|\eta|<\delta$ and $\tau,\tau+\eta \in(0,r_1],$ we get
\begin{align}
&\mathbb{E}\|(\Upsilon \zeta^n)(\tau+\eta)-(\Upsilon \zeta^n)(\tau)\|^2\ \nonumber \\&\leq \ 5 \Big[\mathbb{E}\|[\mathcal{T}_q(\tau+\eta)-\mathcal{T}_q(\tau)][-k_1(\overline{\uppsi}+\zeta^n)-b(0,\uppsi)]\|^2+\mathbb{E}\|[\mathcal{S}_q(\tau+\eta)-\mathcal{S}_q(\tau)][\xi_1-k_2(\overline{\uppsi}+\zeta^n)]\|^2 \nonumber \\&\qquad+\mathbb{E}\|b(\tau+\eta,\overline{\uppsi}_{\varrho(\tau+\eta,\ \overline{\uppsi}_{\tau+\eta}+\zeta_{\tau+\eta}^n)}+\zeta_{\varrho(\tau+\eta,\ \overline{\uppsi}_{\tau+\eta}+\zeta_{\tau+\eta}^n)}^n) -b(\tau,\overline{\uppsi}_{\varrho(\tau,\overline{\uppsi}_\tau+\zeta_\tau^n)}+\zeta_{\varrho(\tau,\overline{\uppsi}_\tau+\zeta_\tau^n)}^n)\|^2\nonumber \\&\qquad+ \mathbb{E}\|\int_0^\tau \int_{\gamma} [\mathcal{S}_q (\tau+\eta-s)-\mathcal{S}_q (\tau-s)] f(s,\vartheta,\overline{\uppsi}_{\varrho(s,\overline{\uppsi}_s+\zeta_s^n)}+\zeta_{\varrho(s,\overline{\uppsi}_s+\zeta_s^n)}^n)\widetilde{Z}(ds,d\vartheta)\nonumber \\&\qquad+\int_\tau^{\tau+\eta} \int_{\gamma} \mathcal{S}_q (\tau+\eta-s) f(s,\vartheta,\overline{\uppsi}_{\varrho(s,\overline{\uppsi}_s+\zeta_s^n)}+\zeta_{\varrho(s,\overline{\uppsi}_s+\zeta_s^n)}^n) \widetilde{Z}(ds,d\vartheta) \|^2 \nonumber \\&\qquad+\mathbb{E}\|\int_0^\tau [\mathcal{T}_q (\tau+\eta-s)-\mathcal{T}_q (\tau-s)] \Big(\int_{-\infty}^s (h(z, \overline{\uppsi}_{\varrho(z,\overline{\uppsi}_{z}+\zeta_{z}^n)}+\zeta_{\varrho(z,\overline{\uppsi}_{z}+\zeta_{z}^n)}^n)d\omega(z)\Big)ds \nonumber \\&\qquad+\int_\tau^{\tau+\eta} \mathcal{T}_q (\tau+\eta-s) \Big(\int_{-\infty}^s (h(z, \overline{\uppsi}_{\varrho(z,\overline{\uppsi}_{z}+\zeta_{z}^n)}+\zeta_{\varrho(z,\overline{\uppsi}_{z}+\zeta_{z}^n)}^n)d\omega(z)\Big)ds\|^2\Big]\nonumber \\ &\quad\leq \ 5 \mathop{\sum}\limits_{j=1}^5 \mathbb{E}\|P_j(t+\eta)-P_j(t)\|^2
\end{align}
Using $(S_1)$, $(S_3)$, it follows that
\begin{align}
\mathbb{E}\|P_1(\tau+\eta)-P_1(\tau)\|^2 \ \leq&\ \|\mathcal{T}_q(\tau+\eta)-\mathcal{T}_q(\tau)\|^2\| \ [k_1(\overline{\uppsi}+\zeta^n)+b(0,\uppsi)]\|^2\nonumber \\ \leq&\ 2\epsilon^2\big[M_b(1+\|\uppsi\|^2)+L_{k_1}\|\overline{\uppsi}+\zeta^n\|_\mathfrak{D}^2+\hat{L}_{k_1}\big] %\\\mathbb{E}\|P_2(\tau+\eta)-P_2(\tau)\|^2 \ \leq&\ \|\mathcal{S}_q(\tau+\eta)-\mathcal{S}_q(\tau)\|^2\mathbb{E}\|\xi_1+k_2(\overline{\uppsi}+\zeta^n)\|^2\nonumber\\ \leq&\ 2 \epsilon^2\big[\|\xi_1\|^2+L_{k_2}\|\overline{\uppsi}+\zeta^n\|_\mathfrak{D}^2+\hat{L}_{k_2}\big]
\end{align}  
% $ \hspace{6.5cm}\rightarrow 0$ as $\eta \rightarrow 0$, for sufficiently small positive value of $\epsilon$.
\begin{align}
\mathbb{E}\|P_2(\tau+\eta)-P_2(\tau)\|^2 \ \leq&\ \|\mathcal{S}_q(\tau+\eta)-\mathcal{S}_q(\tau)\|^2\mathbb{E}\|\xi_1-k_2(\overline{\uppsi}+\zeta^n)\|^2\nonumber\\ \leq&\ 2 \epsilon^2\big[\|\xi_1\|^2+L_{k_2}\|\overline{\uppsi}+\zeta^n\|_\mathfrak{D}^2+\hat{L}_{k_2}\big]
\end{align}
%$ \hspace{7cm}\rightarrow 0$ as $\eta \rightarrow 0$, for sufficiently small value of $\epsilon$.
\begin{align}
&\mathbb{E}\|P_3(\tau+\eta)-P_3(\tau)\|^2 \nonumber\\& \leq\  \|b(\tau+\eta,\overline{\uppsi}_{\varrho(\tau+\eta,\ \overline{\uppsi}_{\tau+\eta}+\zeta_{\tau+\eta}^n)}+\zeta_{\varrho(\tau+\eta,\ \overline{\uppsi}_{\tau+\eta}+\zeta_{\tau+\eta}^n)}^n) -b(\tau,\overline{\uppsi}_{\varrho(\tau,\overline{\uppsi}_\tau+\zeta_\tau)}+\zeta_{\varrho(\tau,\overline{\uppsi}_\tau+\zeta_\tau)})\|^2
\end{align}
%$\hspace{1.5cm}\rightarrow\  0 \ as \ \eta \rightarrow 0$, for small positive number $\epsilon$.\\
By the Lebesgue dominated convergence theorem, $(S_7)$, and Remark $3.3.2$ in \cite{Zhuthesis}, we obtain
\begin{align}
&\mathbb{E}\|P_4(\tau+\eta)-P_4(\tau)\|^2 \nonumber\\ &\leq\ 2\mathbb{E}\|\int_0^\tau \int_{\gamma} [\mathcal{S}_q (\tau+\eta-s)-\mathcal{S}_q (\tau-s)] f(s,\vartheta,\overline{\uppsi}_{\varrho(s,\overline{\uppsi}_s+\zeta_s^n)}+\zeta_{\varrho(s,\overline{\uppsi}_s+\zeta_s^n)}^n\widetilde{Z}(ds,d\vartheta)\|^2\nonumber \\ 
&\quad+2\mathbb{E}\|\int_\tau^{\tau+\eta} \int_{\gamma} \mathcal{S}_q (\tau+\eta-s)f(s,\vartheta,\overline{\uppsi}_{\varrho(s,\overline{\uppsi}_s+\zeta_s^n)}+\zeta_{\varrho(s,\overline{\uppsi}_s+\zeta_s^n)}^n \widetilde{Z}(ds,d\vartheta) \|^2\nonumber \\ \ &\leq\ 2\mathbb{E}\int_0^\tau \int_{\gamma} \|[\mathcal{S}_q (\tau+\eta-s)-\mathcal{S}_q (\tau-s)] f(s,\vartheta,\overline{\uppsi}_{\varrho(s,\overline{\uppsi}_s+\zeta_s^n)}+\zeta_{\varrho(s,\overline{\uppsi}_s+\zeta_s^n)}^n)\|^2\varkappa(d\vartheta)ds\nonumber \\&\quad+2\int_\tau^{\tau+\eta} \int_{\gamma} \mathbb{E}\|\mathcal{S}_q (\tau+\eta-s)f(s,\vartheta,\overline{\uppsi}_{\varrho(s,\overline{\uppsi}_s+\zeta_s^n)}+\zeta_{\varrho(s,\overline{\uppsi}_s+\zeta_s^n)}^n) \|^2\varkappa(d\vartheta)ds\nonumber \\ &\leq\ 2\int_0^\tau \int_{\gamma} \|\mathcal{S}_q (\tau+\eta-s)-\mathcal{S}_q (\tau-s)\|^2\mathbb{E}\| f(s,\vartheta,\overline{\uppsi}_{\varrho(s,\overline{\uppsi}_s+\zeta_s^n)}+\zeta_{\varrho(s,\overline{\uppsi}_s+\zeta_s^n)}^n)\|^2\varkappa(d\vartheta)ds\nonumber \\&\quad+2\mathrm{M}^2\int_\tau^{\tau+\eta} \int_{\gamma} \mathbb{E}\|f(s,\vartheta,\overline{\uppsi}_{\varrho(s,\overline{\uppsi}_s+\zeta_s^n)}+\zeta_{\varrho(s,\overline{\uppsi}_s+\zeta_s^n)}^n) \|_\mathfrak{D}^2\varkappa(d\vartheta)ds\nonumber \\ &\leq\ 2\epsilon^2\int_0^\tau \mathfrak{n}(s)\mathcal{A}_f(r^*) ds +2\mathrm{M}^2\int_\tau^{\tau+\eta}\mathfrak{n}(s)\mathcal{A}_f(r^*) ds
\end{align} 
%$\hspace{1.5cm} \rightarrow 0$ as $|\eta| \rightarrow 0$, for sufficiently small positive number $\epsilon$.
\begin{align}
&\mathbb{E}\|P_5(\tau+\eta)-P_5(\tau)\|^2 \ \nonumber\\&\leq\ 2\mathbb{E}\|\int_0^\tau[\mathcal{T}_q (\tau+\eta-s)-\mathcal{T}_q (\tau-s)]\Big(\int_{-\infty}^s h(z, \overline{\uppsi}_{\varrho(z,\overline{\uppsi}_{z}+\zeta_{z}^n)}+\zeta_{\varrho(z,\overline{\uppsi}_{z}+\zeta_{z}^n)}^n)d\omega(z)\Big)ds\|^2\nonumber\\&\qquad+2\mathbb{E}\|\int_\tau^{\tau+\eta} \mathcal{T}_q (\tau+\eta-s) \Big(\int_{-\infty}^s h(z, \overline{\uppsi}_{\varrho(z,\overline{\uppsi}_{z}+\zeta_{z}^n)}+\zeta_{\varrho(z,\overline{\uppsi}_{z}+\zeta_{z}^n)}^n)d\omega(z)\Big)ds\|^2 \nonumber \\ &\leq\ 2\epsilon^2 r_1\int_0^\tau [2M_h+2Tr(\mathtt{Q})\mathfrak{m}(s)\mathcal{A}_h(\|\overline{\uppsi}_{\varrho(z,\overline{\uppsi}_{z}+\zeta_{z}^n)}+\zeta_{\varrho(z,\overline{\uppsi}_{z}+\zeta_{z}^n)}^n\|_\mathfrak{D}^2)]ds\nonumber\\&\quad+2\mathrm{M}^2\int_\tau^{\tau+\eta}\big[2M_h+2Tr(\mathtt{Q})\mathfrak{m}(s)\mathcal{A}_h(\|\overline{\uppsi}_{\varrho(z,\overline{\uppsi}_{z}+\zeta_{z}^n)}+\zeta_{\varrho(z,\overline{\uppsi}_{z}+\zeta_{z}^n)}^n\|_\mathfrak{D}^2)\big]ds \nonumber \\ &\leq\ 2\epsilon^2 r_1^2 [2M_h+2Tr(\mathtt{Q})\mathcal{A}_h(r^*)\mathop{\sup}\limits_{s\in(0,r_1]}\mathfrak{m}(s)]+2\mathrm{M}^2\eta\int_\tau^{\tau+\eta}\big[2M_h+2Tr(\mathtt{Q})\mathcal{A}_h(r^*)\mathfrak{m}(s)\big]ds 
\end{align}
%$\qquad \qquad \ \rightarrow 0$ as $|\eta| \rightarrow 0$ for sufficiently small value of $\epsilon$.\\
Thus from $(3.14)$ to $(3.18)$, observe that\\
%\begin{align*}
$\mathbb{E}\|(\Upsilon \zeta^n)(\tau+\eta)-(\Upsilon \zeta^n)(\tau)\|^2\ \rightarrow\ 0$ as $\eta \rightarrow 0$, for each $\zeta^n \in \overline{\mathtt{B}}_\alpha$ and sufficiently small $\epsilon>0$.\\ 
Next, for  $\tau,\tau+\eta \in\mathop{\cup}\limits_{i=1}^\mathcal{N} (r_i,s_i],\ 0<|\eta|<\delta$, using $(S_3)$ and $(S_4)$, we obtain
\begin{align*}
&\mathbb{E}\|(\Upsilon \zeta^n)(\tau+\eta)-(\Upsilon \zeta^n)(\tau)\|^2 \ \\&\leq\ 2\mathbb{E}\|m_i(\tau+\eta,\overline{\uppsi}_{\varrho(\tau+\eta,\overline{\uppsi}_{\tau+\eta}+\zeta_{\tau+\eta}^n)}+\zeta_{\varrho(\tau+\eta,\overline{\uppsi}_{\tau+\eta}+\zeta_{\tau+\eta}^n)}^n)-m_i(\tau,\overline{\uppsi}_{\varrho(\tau,\overline{\uppsi}_\tau+\zeta_\tau^n)}+\zeta_{\varrho(\tau,\overline{\uppsi}_\tau+\zeta_\tau^n)}^n)\|^2+2\epsilon^2N_1^2\|\uppsi\|_\mathfrak{D}^2
\end{align*} $  \ \rightarrow 0$ as $\eta \rightarrow 0$.\\
Analogously, for  $\tau,\tau+\eta \in\mathop{\cup}\limits_{i=1}^\mathcal{N} (s_i,r_{i+1}],\ 0<|\eta|<\delta$, using $(S_1)$, $(S_4)$, we get
\begin{align*}
&\mathbb{E}\|(\Upsilon \zeta^n)(\tau+\eta)-(\Upsilon \zeta^n)(\tau)\|^2 \ \\&\leq\ 7\Big[\mathbb{E}\|[\mathcal{T}_q (\tau+\eta-s)-\mathcal{T}_q (\tau-s)][ l_i(\overline{\uppsi}_{s_i}+\zeta_{s_i}^n) + m_i(s_i,\overline{\uppsi}_ {\varrho(s_i,\overline{\uppsi}_{s_i}+\zeta_{s_i}^n)}+\zeta_{\varrho(s_i,\overline{\uppsi}_{s_i}+\zeta_{s_i}^n)}^n) \\&\qquad- b(s_i,\overline{\uppsi}_{\varrho(s_i,\overline{\uppsi}_{s_i}+\zeta_{s_i}^n)}+\zeta_{\varrho(s_i,\overline{\uppsi}_{s_i}+\zeta_{s_i}^n)}^n)]\|^2+\mathbb{E}\|[\mathcal{T}_q(\tau+\eta)-\mathcal{T}_q(\tau)]\uppsi(0)\|^2\\&\qquad+\mathbb{E}\| b(\tau+\eta,\overline{\uppsi}_{\varrho(\tau+\eta,\overline{\uppsi}_{\tau+\eta}+\zeta_{\tau+\eta}^n)}+\zeta_{\varrho(\tau,\overline{\uppsi}_\tau+\zeta_{\tau+\eta}^n)}^n)-b(\tau,\overline{\uppsi}_{\varrho(\tau,\overline{\uppsi}_\tau+\zeta_\tau^n)}+\zeta_{\varrho(\tau,\overline{\uppsi}_\tau+\zeta_\tau^n)}^n)\|^2\\&\qquad+ \mathbb{E}\| \int_{s_i}^\tau \int_{\gamma} [\mathcal{S}_q (\tau+\eta-s)-\mathcal{S}_q (\tau-s)] f(s,\vartheta,\overline{\uppsi}_{\varrho(\tau,\overline{\uppsi}_s+\zeta_s^n)}+\zeta_{\varrho(s,\overline{\uppsi}_s+\zeta_s^n)}^n) \widetilde{Z}(ds,d\vartheta)\|^2 \\&\qquad+ \mathbb{E}\| \int_\tau^{\tau+\eta} \int_{\gamma} \mathcal{S}_q (\tau+\eta-s) f(s,\vartheta,\overline{\uppsi}_{\varrho(s,\overline{\uppsi}_s+\zeta_s^n)}+\zeta_{\varrho(s,\overline{\uppsi}_s+\zeta_s^n)}^n) \widetilde{Z}(ds,d\vartheta)\|^2 \\&\qquad+ \mathbb{E}\| \int_{s_i}^\tau [\mathcal{T}_q (\tau+\eta-s)-\mathcal{T}_q (\tau-s)] \Big(\int_{-\infty}^s h(z,\overline{\uppsi}_{\varrho(z,\overline{\uppsi}_{z}+\zeta_{z}^n)}+\zeta_{\varrho(z,\overline{\uppsi}_{z}+\zeta_{z}^n)}^n)d\omega(z)\Big)ds\|^2\\&\qquad+ \mathbb{E}\| \int_\tau^{\tau+\eta} \mathcal{T}_q (\tau+\eta-s) \Big(\int_{-\infty}^s h(z,\overline{\uppsi}_{\varrho(z,\overline{\uppsi}_{z}+\zeta_{z}^n)}+\zeta_{\varrho(z,\overline{\uppsi}_{z}+\zeta_{z}^n)}^n)d\omega(z)\Big)ds\|^2\Big]
 \\&\ \leq \ 21\|\mathcal{T}_q (\tau+\eta-s)-\mathcal{T}_q (\tau-s)\|^2[\mathbb{E}\| l_i(\overline{\uppsi}_{s_i}+\zeta_{s_i}^n)\|^2 +\mathbb{E}\| m_i(s_i,\overline{\uppsi}_ {\varrho(s_i,\overline{\uppsi}_{s_i}+\zeta_{s_i}^n)}+\zeta_{\varrho(s_i,\overline{\uppsi}_{s_i}+\zeta_{s_i}^n)}^n)\|^2 \\&\qquad+\mathbb{E}\| b(s_i,\overline{\uppsi}_{\varrho(s_i,\overline{\uppsi}_{s_i}+\zeta_{s_i}^n)}+\zeta_{\varrho(s_i,\overline{\uppsi}_{s_i}+\zeta_{s_i}^n)}^n)\|^2]+7\|\mathcal{T}_q(\tau+\eta)-\mathcal{T}_q(\tau)\|^2N_1^2\|\uppsi\|_\mathfrak{D}^2\\&\qquad+7\mathbb{E}\| b(\tau+\eta,\overline{\uppsi}_{\varrho(\tau+\eta,\overline{\uppsi}_{\tau+\eta}+\zeta_{\tau+\eta}^n)}+\zeta_{\varrho(\tau,\overline{\uppsi}_\tau+\zeta_{\tau+\eta}^n)}^n)-b(\tau,\overline{\uppsi}_{\varrho(\tau,\overline{\uppsi}_\tau+\zeta_\tau^n)}+\zeta_{\varrho(\tau,\overline{\uppsi}_\tau+\zeta_\tau^n)}^n)\|^2 \\&\qquad+7 \mathbb{E} \int_{s_i}^\tau \int_{\gamma} \|[\mathcal{S}_q (\tau+\eta-s)-\mathcal{S}_q (\tau-s)]f(s,\vartheta,\overline{\uppsi}_{\varrho(\tau,\overline{\uppsi}_s+\zeta_s^n)}+\zeta_{\varrho(s,\overline{\uppsi}_s+\zeta_s^n)}^n)\|^2 \varkappa(d\vartheta)ds \\&\qquad+7 \mathbb{E} \int_\tau^{\tau+\eta} \int_{\gamma} \|\mathcal{S}_q (\tau+\eta-s) f(s,\vartheta,\overline{\uppsi}_{\varrho(s,\overline{\uppsi}_s+\zeta_s^n)}+\zeta_{\varrho(s,\overline{\uppsi}_s+\zeta_s^n)}^n)\|^2 \varkappa(d\vartheta)ds \\&\qquad+7  \int_{s_i}^\tau \mathbb{E}\|[\mathcal{T}_q (\tau+\eta-s)-\mathcal{T}_q (\tau-s)] \Big(\int_{-\infty}^s h(z,\overline{\uppsi}_{\varrho(z,\overline{\uppsi}_{z}+\zeta_{z}^n)}+\zeta_{\varrho(z,\overline{\uppsi}_{z}+\zeta_{z}^n)}^n)d\omega(z)\Big)\|^2ds \\&\qquad+7 \int_\tau^{\tau+\eta} \mathbb{E}\| \mathcal{T}_q (\tau+\eta-s)\Big(\int_{-\infty}^s h(z,\overline{\uppsi}_{\varrho(z,\overline{\uppsi}_{z}+\zeta_{z}^n)}+\zeta_{\varrho(z,\overline{\uppsi}_{z}+\zeta_{z}^n)}^n)d\omega(z)\Big)\|^2ds\\
 &\ \leq\ 21\epsilon^2\big[\mathcal{A}_{l_i}(r^{**})+M_i(1+r^*)+M_b(1+r^*)\big]+7\epsilon^2 \|\uppsi\|_\mathfrak{D}^2+7\epsilon^2 \int_{s_i}^\tau \mathfrak{n}(s)\mathcal{A}_f(r^*) ds\\&\qquad+7\mathrm{M}^2  \int_\tau^{\tau+\eta} \mathfrak{n}(s)\mathcal{A}_f(r^*) ds+7 \epsilon^2(r_{i+1}-s_i) \int_{s_i}^\tau[2M_h+2Tr(\mathtt{Q})\mathcal{A}_h(r^*)\mathfrak{m}(s)]ds\\&\qquad+7\mathrm{M}^2\eta \int_\tau^{\tau+\eta}[2M_h+2Tr(\mathtt{Q})\mathcal{A}_h(r^*)\mathfrak{m}(s)]ds
\end{align*} $\qquad \rightarrow 0$ as $\eta \rightarrow 0$ independently of $\zeta^n$.\\
Hence $\{\Upsilon \zeta^n\}_{n=1}^\infty$ is equicontinuous on $\mathcal{I}$. Consequently, $\Upsilon(\mathbf{G})$ is equicontinuous on $\mathcal{I}$.

Next, to prove $\overline{\mathbf{G}}$ is compact, it suffices to prove that $\mu(\mathbf{G})=0$. For  $\tau\in [0,r_1]$, using hypothesis, Lemma \ref{lemma2.2} and \ref{lemma2.3}, we obtain
\begin{align*}
&\mu(\Upsilon\{\zeta^n(\tau)\}_{n=1}^\infty)\ \\&\leq\ \mu(\{\mathcal{T}_q(\tau)k_1(\overline{\uppsi}+\zeta^n)\}_{n=1}^\infty)+\mu(\{\mathcal{S}_q(\tau)k_2(\overline{\uppsi}+\zeta^n)\}_{n=1}^\infty)+\mu(\{b(\tau,\overline{\uppsi}_{\varrho(\tau,\overline{\uppsi}_\tau+\zeta_\tau^n)}+\zeta_{\varrho(\tau,\overline{\uppsi}_\tau+\zeta_\tau^n)}^n)\}_{n=1}^\infty)\\&\quad+\mu(\{\int_0^\tau \int_{\gamma} \mathcal{S}_q (\tau-s) f(s,\vartheta,\overline{\uppsi}_{\varrho(s,\overline{\uppsi}_s+\zeta_s^n)}+\zeta_{\varrho(s,\overline{\uppsi}_s+\zeta_s^n)}^n\widetilde{Z}(ds,d\vartheta)\}_{n=1}^\infty)\\&\quad+\mu(\{\int_0^\tau \mathcal{T}_q (\tau-s)\big(\int_{-\infty}^s h(z, \overline{\uppsi}_{\varrho(z,\overline{\uppsi}_{z}+\zeta_{z)}^n}+\zeta_{\varrho(z,\overline{\uppsi}_{z}+\zeta_{z)}^n}^n)d\omega(z)\big)ds\}_{n=1}^\infty)\\\ &\leq\ \mathrm{M} (l_1^*+l_2^*) \mathop{\sup}\limits_{-\infty<\theta\leq0}\mu(\{\overline{\uppsi}(\tau+\theta)+\zeta^n(\tau+\theta)\}_{n=1}^\infty)+l_b(\tau)\mathop{\sup}\limits_{-\infty<\theta\leq0}\mu(\{\overline{\uppsi}(\tau+\theta)+\zeta^n(\tau+\theta)\}_{n=1}^\infty)\\&\quad+2\mathrm{M} \int_0^\tau \int_{\gamma}l_f(s,\vartheta)\mathop{\sup}\limits_{-\infty<\theta\leq0}\mu\big(\{\overline{\uppsi}(s+\theta)+\zeta^n(s+\theta)\}_{n=1}^\infty)\widetilde{Z}(ds,d\vartheta)\\&\quad+2\mathrm{M}\int_0^\tau\big[\mu(k(\tau)+\{\int_0^s h(z, \overline{\uppsi}_{\varrho(z,\overline{\uppsi}_{z}+\zeta_{z)}^n}+\zeta_{\varrho(z,\overline{\uppsi}_{z}+\zeta_{z)}^n}^n)d\omega(z))\}_{n=1}^\infty\big)\big]ds\\ \ &\leq\ \mathrm{M} (l_1^*+l_2^*) \mathop{\sup}\limits_{-\infty<\theta\leq0}\mu(\{\overline{\uppsi}(\tau+\theta)+\zeta^n(\tau+\theta)\}_{n=1}^\infty)+l_b(\tau)\mathop{\sup}\limits_{-\infty<\theta\leq0}\mu(\{\overline{\uppsi}(\tau+\theta)+\zeta^n(\tau+\theta)\}_{n=1}^\infty)\\&\quad+4\mathrm{M}\int_0^\tau \int_\gamma l_f(s,\vartheta)\mathop{\sup}\limits_{-\infty<\theta\leq0}\mu(\{\overline{\uppsi}(s+\theta)+\zeta^n(s+\theta)\}_{n=1}^\infty\widetilde{Z}(ds,d\vartheta)\\&\quad+4\mathrm{M}[r_1(Tr(\mathtt{Q}))]^{1/2}\big(\int_0^\tau\big\{\int_0^s \Big(l_h(z) \mathop{\sup}\limits_{-\infty<\theta\leq0}\mu(\{\overline{\uppsi}(z+\theta)+\zeta^n(z+\theta)\big\}_{n=1}^\infty\Big)dz\}^2ds\big)^{1/2}  \\ \ &\leq\ \big[ \mathrm{M} (l_1^*+l_2^*)+l_b^*+4\mathrm{M} \mho +4\mathrm{M}\sqrt{r_1}\sqrt{Tr(\mathtt{Q})}\|\chi\|_{L^2}\big]\mathop{\sup}\limits_{0\leq z\leq \tau}\mu(\{\zeta^n(z)\}_{n=1}^\infty),
\end{align*}
where $\chi(s) = \int_0^s l_h(z)dz$.\\ 
Similarly, for $\tau\in \mathop{\cup}\limits_{i=1}^\mathcal{N}(r_i,s_i]$,
\begin{align*}
&\mu(\Upsilon\{\zeta^n(\tau)\}_{n=1}^\infty)\ \\&\leq\ \mu(\{l_i(\overline{\uppsi}_{r_i}+\zeta_{r_i}^n)\}_{n=1}^\infty)+\mu(\{m_i(\tau,\overline{\uppsi}_{\varrho(\tau,\overline{\uppsi}_\tau+\zeta_\tau^n)}+\zeta_{\varrho(\tau,\overline{\uppsi}_\tau+\zeta_\tau^n)}^n)\}_{n=1}^\infty)\\ \ &\leq\ L_i \mathop{\sup}\limits_{-\infty<\theta\leq0}\mu(\{\overline{\uppsi}(r_i+\theta)+\zeta^n(r_i+\theta)\}_{n=1}^\infty)+l_{m_i}(\tau)\mathop{\sup}\limits_{-\infty<\theta\leq0}\mu(\{\overline{\uppsi}(\tau+\theta)+\zeta^n(\tau+\theta)\}_{n=1}^\infty)\\ \ &\leq\ (L_i+l_{m_i}(\tau)) \mathop{\sup}\limits_{r_i\leq z\leq \tau}\mu(\{\zeta^n(z)\}_{n=1}^\infty)\\ \ &\leq \ (L_i+l_{m_i}^*)\ \mu(\{\zeta^n\}_{n=1}^\infty)
\end{align*}
Now, for $\tau\in \mathop{\cup}\limits_{i=1}^\mathcal{N}(s_i,r_{i+1}]$,
\begin{align*}
&\mu(\Upsilon\{\zeta^n(\tau)\}_{n=1}^\infty)\ \\&\leq\ \mu(\mathcal{T}_q(\tau-s_i)\{ l_i(\overline{\uppsi}_{s_i}+\zeta_{s_i}^n) + m_i(s_i,\overline{\uppsi}_ {\varrho(s_i,\overline{\uppsi}_{s_i}+\zeta_{s_i}^n)}+\zeta_{\varrho(s_i,\overline{\uppsi}_{s_i}+\zeta_{s_i}^n)}^n) \\&\quad- b(s_i,\overline{\uppsi}_{\varrho(s_i,\overline{\uppsi}_{s_i}+\zeta_{s_i}^n)}+\zeta_{\varrho(s_i,\overline{\uppsi}_{s_i}+\zeta_{s_i}^n)}^n)\}_{n=1}^\infty)+\mu( \{b(\tau,\overline{\uppsi}_{\varrho(\tau+\eta,\overline{\uppsi}_{\tau+\eta}+\zeta_{\tau+\eta}^n)}+\zeta_{\varrho(\tau,\overline{\uppsi}_\tau+\zeta_{\tau+\eta}^n)}^n)\}_{n=1}^\infty)\\&\quad+\mu\big(\big\{\int_{s_i}^\tau \int_{\gamma} \mathcal{S}_q (\tau-s) f(s,\vartheta,\overline{\uppsi}_{\varrho(s,\overline{\uppsi}_s+\zeta_s^n)}+\zeta_{\varrho(s,\overline{\uppsi}_s+\zeta_s^n)}^n) \widetilde{Z}(ds,d\vartheta)\big\}_{n=1}^\infty\big)\\&\quad+ \mu\big(\big\{\int_{s_i}^\tau \mathcal{T}_q (\tau-s) \big(\int_{-\infty}^s h(z,\overline{\uppsi}_{\varrho(z,\overline{\uppsi}_z+\zeta_z^n)}+\zeta_{\varrho(z,\overline{\uppsi}_{z}+\zeta_z^n)}^n d\omega(z)\big)ds\big\}_{n=1}^\infty\big)
\end{align*}
Using similar arguments as above, we have
\begin{align*}
&\mu(\Upsilon\{\zeta^n(\tau)\}_{n=1}^\infty)\ \\&\leq\ \Big[\mathrm{M}\big(L_i+l_{m_i}(\tau)+l_b(\tau)\big)+l_b(\tau)+4\mathrm{M}\int_{s_i}^\tau \int_\gamma l_f(s,\vartheta)\widetilde{Z}(ds,d\vartheta)\\&\qquad+4\mathrm{M}[(r_{i+1}-s_i)(Tr(\mathtt{Q}))]^{1/2}\left\{\int_{s_i}^\tau\int_0^s l_h^2(z)dz ds \right\}^{1/2}\Big] \mathop{\sup}\limits_{-\infty<\theta\leq0}\mu(\{\overline{\uppsi}(\tau+\theta)+\zeta^n(\tau+\theta)\}_{n=1}^\infty\\ \ &\leq\ \big[ \mathrm{M} (L_i+l_{m_i}^*+l_b^*)+l_b^*+4\mathrm{M}\mho+4\mathrm{M}\sqrt{(r_{i+1}-s_i)}\sqrt{Tr(\mathtt{Q})}\|\chi\|_{L^2}\mathop{\sup}\limits_{s_i<z\leq \tau}\mu(\{\zeta^n(z)\}_{n=1}^\infty)
\end{align*}
Therefore, for all $\tau\in\mathcal{I}$,
\begin{equation*}
\begin{aligned}
\mu(\Upsilon\{\zeta^n(\tau)\}_{n=1}^\infty)\ &\leq \ \mathop{\max}\limits_{0\leq i\leq \mathcal{N}}\Big\{l_1^*+l_2^*+L_i+l_{m_i}^*+ \mathrm{M}(L_i+l_{m_i}^*+l_b^*)+l_b^*+4\mathrm{M}\mho\\&\qquad+4\mathrm{M}\sqrt{a}\sqrt{Tr(\mathtt{Q})}\|\chi\|_{L^2} \Big\}  \mathop{\sup}\limits_{0\leq z\leq \tau}\mu(\{\zeta^n(z)\}_{n=1}^\infty)\\ &=\   \triangle_2\  \mu(\{\zeta^n\}_{n=1}^\infty).
\end{aligned}
\end{equation*}
%where $\triangle_2 = \mathop{\max}\limits_{0\leq i\leq N}\Big\{l_1^*+l_2^*+L_i+l_{m_i}^*+ \mathrm{M}(L_i+l_{m_i}^*+l_b^*)+l_b^*+4\mathrm{M}\mho+4\mathrm{M}\sqrt{a}\sqrt{(Tr(\mathtt{Q}))}\|\chi\|_{L^2} \Big\}\ <\ 1,$\\ \\
Thus, $ \mu(\Upsilon(\mathbf{G})) \ \leq\ \triangle_2\  \mu(\mathbf{G})$.

 Now, Mönch’s condition implies that $\mu(\mathbf{G})\leq \mu( \overline{\textbf{co}}(\{0\}\cup \Upsilon(\mathbf{G})))=\mu(\Upsilon(\mathbf{G}))\leq \triangle_2\  \mu(\mathbf{G})$, which shows by the Inequality (\ref{3.1}) that $\mu(\mathbf{G})=0$.
Now, by Lemma \ref{lemma2.4}, we conclude that $\Upsilon$ has a fixed point $\zeta^*$ in $\overline{\mathtt{B}}_\alpha$. 
Hence $\xi(\tau)= \bar{\zeta}^*(\tau)+\overline{\uppsi}(\tau),\  \tau \in (-\infty,a]$, is a mild solution of the system (\ref{mainequation}).
\end{proof}
\section{An Example}\label{Sect.4}
%In this sect., we construct an example and apply Theorem \ref{thm3.1}, to show the existence of a solution.
We construct the following nonlocal neutral fractional stochastic differential equation with NIIs and SDD:
\begin{equation}
\begin{cases}
%\begin{aligned}
^CD_\tau^q\Big[\xi(\tau,\upsilon)-\int_{-\infty}^\tau g_1(\tau,s-\tau,\upsilon,\xi(s-\sigma_1(\tau)\sigma_2(\|\xi(\tau)\|),\upsilon))ds\Big]\\=\frac{\partial^2}{\partial \upsilon^2}\Big[\xi(\tau,\upsilon)-\int_{-\infty}^\tau g_1(\tau,s-\tau,\upsilon,\xi(s-\sigma_1(\tau)\sigma_2(\|\xi(\tau)\|),\upsilon))ds\Big]\\\ +\frac{1}{\Gamma{(2-q)}}\int_0^\tau(\tau-s)^{1-q}\big[\int_\gamma \vartheta \int_{-\infty}^s g_2(s,z-s,\upsilon \xi(z-\sigma_1(s)\sigma_2(\|\xi(s)\|),\upsilon))dz\Big]\widetilde{Z}(ds,d\vartheta)\\ \ +\frac{1}{\Gamma{(1-q)}}\int_0^\tau(\tau-s)^{-q}\Big[\int_{-\infty}^\tau\Big( \int_{-\infty}^s g_3(s,z-s,\upsilon,\xi(z-\sigma_1(s)\sigma_2(\|\xi(s)\|),\upsilon))dz\Big)d\omega(s)\Big]ds, \\ \qquad (\tau,\upsilon)\in \mathop{\cup}\limits_{i=1}^\mathcal{N} (s_i,r_{i+1}] \times [0,\pi]; \\
 \xi(\tau,\upsilon)= \int_{-\infty}^{r_i}a_i(s-r_i)\xi(s,\upsilon)ds+\int_{-\infty}^\tau \tilde{a}_i(\tau,s-\tau,\upsilon,\xi(s-\sigma_1(\tau)\sigma_2(\|\xi(\tau)\|),\upsilon))ds,\\ \qquad (\tau,\upsilon)\in \mathop{\cup}\limits_{i=1}^\mathcal{N}(r_i,s_i]\times [0,\pi];
%\end{aligned}
\\
\xi(\tau,0)=\xi(\tau,\pi)=0,\quad \xi^\prime(\tau,0)=\xi^\prime(\tau,\pi)=0, \ \tau\in (0,a];\\
 \xi(\tau,\upsilon)=\uppsi(\tau,\upsilon)\in \mathfrak{D},\ \xi^\prime(\tau,\upsilon)= \uppsi_1(\tau,\upsilon)\in \mathcal{V},\ (\tau,\upsilon)\in (-\infty,0)\times[0,\pi];\\
 \xi(0,\upsilon)+\int_0^\pi f_1(\upsilon,z)\xi(\tau,z)dz= \xi_0\in \mathfrak{D},\ (\tau,\upsilon)\in (-\infty,0)\times[0,\pi];\\ 
[\xi(\tau,\upsilon)-g_1(\tau,\xi(\tau-\sigma_1(\tau)\sigma_2(\|\xi(\tau)\|)),\upsilon)]_{\tau=0}^\prime+\int_0^\pi f_2(\upsilon,z)\xi(\tau,z)dz= \xi_0^\prime \in \mathcal{V},\\ \  (\tau,\upsilon)\in \mathcal{I}_0\times[0,\pi];\\
 [\xi(\tau,\upsilon)-g_1(\tau,\xi(\tau-\sigma_1(\tau)\sigma_2(\|\xi(\tau)\|)),\upsilon)]_{\tau=s_i}^\prime=0,
\end{cases}
\end{equation}
where $1<q<2,\ 0=r_0=s_0<r_1<s_1<\cdots<s_\mathcal{N}<r_{\mathcal{N}+1}=a$ are prefixed. The mappings $\varrho_1,\varrho_2:[0,\infty)\rightarrow[0,\infty),$ $f_1,f_2:[0,\pi]\times[0,\pi]\rightarrow \mathds{R}$, $a_i:\mathds{R}\rightarrow \mathds{R}$, $\tilde{a}_i:\mathds{R}^4\rightarrow \mathds{R}, \    1\leq~i~\leq~\mathcal{N}$ and $g_j:\mathds{R}^4\rightarrow \mathds{R},\  1\leq j\leq3$  are continuous. Here $\omega(\cdot)$ and $ \widetilde{Z}(ds,d\vartheta)$ denote the standard one dimensional $\mathtt{Q}$-Wiener process and compensating martingale measure, respectively.
% with finite trace defined on the space $(\Omega,\mathscr{F},\{\mathscr{F}_\tau\}_{\tau\geq 0},\mathbb{P})$. Let $\{\tilde{z}(\tau) : \tau \in \mathcal{I}\}$ be the Poisson point process that take values in $\mathds{R}\cup{\{0\}}$ with a $\sigma$-finite intensity measure $\varkappa(d\vartheta)$ on the complete probability space $(\Omega,\mathscr{F},\mathbb{P})$ and $\int_\gamma \vartheta^2 \varkappa(d\vartheta)<\infty$, and is independent of $\omega(\tau))$. Then the Poisson counting measure $Z(d\tau,d\vartheta)$ is produced by $\tilde{z}(\cdot)$ and the compensating martingale measure is defined by $\widetilde{Z}(ds,d\vartheta)=Z(ds,d\vartheta)-\varkappa(d\vartheta)ds$.

 Let $\hat{a} \in \mathds{R}^+\cup\{0\}$, and  $\widetilde{y}: (-\infty,-\hat{a}]\rightarrow \mathds{R}^+\cup\{0\}$ be a measurable function such that $(g$-$5)$--$(g$-$7)$ introduced in \cite{Hino} hold. Define $\mathtt{PC}_{\hat{a}} \times  L^p(\widetilde{y},\mathcal{V}) =\big\{\Psi:\mathcal{I}_0\rightarrow \mathcal{V}/ \ \Psi|_{[-\hat{a},0]} \in  \mathtt{PC}([-\hat{a},0],\mathcal{V}),\ \Psi(\cdot) $ is Lebesgue measurable on $(-\infty,-\hat{a})$ and
%\begin{align*}
$\|\Psi\|_\mathfrak{D}= \mathop{\sup}\limits_{-\hat{r}\leq s\leq 0}\|\Psi(s)\|+\Big(\int_{-\infty}^{-\hat{a}}\widetilde{y}(s)\|\Psi(s)\|^p ds\Big)^\frac{1}{p},\ p\geq1\big\}$.
%\end{align*} 
Clearly, $\mathfrak{D}= \mathtt{PC}_0 \times L^2(\widetilde{y},\mathcal{V})$ satisfies Axioms $\rom{1}-\rom{3}$ with $N_1=1,\  N_2(\tau)=1+(\int_{-\tau}^0\widetilde{y}(s)ds)^{\frac{1}{2}}, \tau~\geq ~0$, and $N_3(\tau)= G(-\tau)^{\frac{1}{2}}$, where G is as defined in $(g$-$6)$.

 Choose $\mathcal{V}=\mathcal{K}=L^2([0,\pi],\mathds{R})$ with the norm $\|\cdot\|$ and $\mathscr{A}:D(\mathscr{A})\subset \mathcal{V}\rightarrow \mathcal{V}$ is defined by $\mathscr{A}\xi=\frac{\partial^2 \xi}{\partial \upsilon^2}$ with the domain $D(\mathscr{A})=\{\xi\in \mathcal{V}:\ \xi,\xi^\prime$ are absolutely continuous, $\xi^{\prime \prime}\in \mathcal{V},\ \xi(0)=\xi(\pi)=0\}$. Evidently, $\mathscr{A}$ is densely defined in $\mathcal{V}$ and it is the infinitesimal generator of the  resolvent family $\{\mathcal{R}_q(\tau)\}_{\tau\geq 0}$ on $\mathcal{V}$. Then there exists   $\mathrm{M}>0$ to ensure that $\mathop{\sup}\limits_{\tau\in \mathcal{I}}\|\mathcal{T}_q(\tau)\|  \vee  \mathop{\sup}\limits_{\tau\in \mathcal{I}}\|\mathcal{S}_q(\tau)\| \leq \mathrm{M}$ (see \cite{Shu&Wang2012}). That is $(S_1)$ holds. For $(\tau,\uppsi)\in \mathcal{I}\times \mathfrak{D}$, set $\uppsi(\varpi)(\upsilon)= \uppsi(\varpi,\upsilon)$, $\xi(\tau)(\upsilon)= \xi(\tau,\upsilon)$ and $\varrho(\tau,\uppsi)=\tau-\sigma_1(\tau)\sigma_2(\|\uppsi(0)\|)$, where $(\varpi,\upsilon)\in \mathcal{I}_0\times [0,\pi]$. The functions $ b: \mathcal{I}\times \mathfrak{D}\rightarrow \mathcal{V},\ f:\mathcal{I}\times \gamma \times \mathfrak{D} \rightarrow \mathcal{V},\ h: \mathcal{I} \times \mathfrak{D} \rightarrow \mathbb{L}_\mathtt{Q}(\mathcal{V}),\ l_i:\mathfrak{D} \rightarrow \mathcal{V},\ m_i:(r_i,s_i]\times \mathfrak{D}\rightarrow \mathcal{V},\ i=1,2,\cdots,\mathcal{N}$ are defined by
\begin{align*}
b(\tau,\uppsi)(\upsilon)&=\int_{-\infty}^0 g_1(\tau,\varpi,\upsilon,\uppsi(\varpi)(\upsilon))d\varpi,\\
 f(\tau,\vartheta,\uppsi)(\upsilon)&= \vartheta\int_{-\infty}^0 g_2(\tau,\varpi,\upsilon,\uppsi(\varpi)(\upsilon))d\varpi,\\
h(\tau,\uppsi)(\upsilon)&= \int_{-\infty}^0 g_3(\tau,\varpi,\upsilon,\uppsi(\varpi)(\upsilon))d\varpi,\\
l_i(\uppsi)(\upsilon)&=\int_{-\infty}^0 a_i(\varpi)\uppsi(\varpi)(\upsilon)d\varpi,\\
m_i(\tau,\uppsi)(\upsilon)&=\int_{-\infty}^0 \tilde{a}_i(\tau,\varpi,\upsilon,\uppsi(\varpi)(\upsilon))d\varpi.
\end{align*}
Thus the system $(4.1)$ can be written in the abstract form (\ref{mainequation}). If $b,f,h,l_i$ and $m_i$ satisfies the assumptions $(S_2)-(S_7)$, then in virtue of Theorem \ref{thm3.1} the system $(4.1)$ admits a mild solution on $\mathcal{I}.$

\section{Conclusion} In literature, only few articles concerned with the existence of a solution in which the resolvent operator is taken to be noncompact. In this article, the existence of a solution for neutral stochastic fractional differential equations of order $q\in(0,2)$ involving NIIs and SDD with the Poisson jumps and the Wiener process has been derived by applying measure of non-compactness via fixed point theorem. An example is constructed to validate the obtained theory. By using the same approach one can extend the obtained result for the given system \ref{mainequation} of order $2<q<3$. In future, we will derive sufficient conditions to guarantee the existence of a solution for the inclusion system associated with (\ref{mainequation}).

\end{document}